\documentclass[12pt]{article}
\usepackage{caption}
\usepackage{url}
\usepackage{color}
\usepackage{amssymb}
\usepackage{amsmath}
\usepackage{amsthm}
\usepackage{graphicx}

\newtheorem{theorem}{Theorem}[section]
\newtheorem{lemma}[theorem]{Lemma}

\newtheorem{definition}[theorem]{Definition}

\newtheorem{corollary}[theorem]{Corollary}
\newtheorem{claim}[theorem]{Claim}
\newtheorem{remark}[theorem]{Remark}

\newtheorem{proposition}[theorem]{Proposition}

\DeclareMathOperator{\adj}{adj}
\DeclareMathOperator{\cof}{cof}

\title{Three Ways to Count Walks in a Digraph {\color{blue}(Extended Version)}}
\author{Matthew P. Yancey \thanks{Institute for Defense Analyses / Center for Computing Sciences (IDA / CCS), mpyance@super.org}}

\begin{document}
\maketitle

\begin{abstract}
We approach the problem of counting the number of walks in a digraph from three different perspectives: enumerative combinatorics, linear algebra, and symbolic dynamics.
\end{abstract}

{\color{blue}This is the extended version of this manuscript.
We have included several results (Theorem \ref{drazin theorem}, Corollary \ref{adjugate theorem}, and associated statements) that are tangential to the core theme.
All of the proofs are self-contained.
We also provide an extended treatment of our examples.}

\section{Introduction}
A walk in a graph is a ``path'' that is described by a sequence of edges, allowing for repeats.
Counting walks of fixed length, especially those with given starting and ending locations, has applications to Markov chains and mixing times \cite{Leven_Peres_Wilmer}, community detection \cite{Gharan_Trevisan}, and quasi-randomness \cite{Chung}.
A \emph{digraph} (directed graph) is a graph with an orientation applied to each edge; a \emph{walk} in a digraph is a walk in the underlying graph such that the orientation for each appearance of each edge is the same orientation as the ``path.''
In this paper we will discuss practical methods for counting walks in a digraph.
The heart of this paper is the application of directed walks to \emph{regular languages}.

In 1958, Chomskey and Miller \cite{Chomskey_Miller} defined and used regular languages as a method to characterize a family of files that are easy to search for.
Regular languages are the foundation behind how search engines and streaming filters operate.
In their original work, Chomskey and Miller proved that for each regular language $L$, there exists a digraph $D$ with vertex sets $I, F \subseteq V$ such that for all $m$ the set of files of length $m$ in $L$ are in bijection with the set of walks of length $m$ that begin in $I$ and terminate in $F$.
For more background on regular languages, see \cite{Hopcroft_Ullman}.

Let us give formal definitions now.
Let $D$ be a finite ground set $V = \{1,2,\ldots,n\}$ called \emph{vertices} and a multi-set $E$ of ordered pairs from the set $V \times V$ called \emph{arcs}.
A $(s,t)$-walk in a digraph is a finite sequence of arcs $e_1, e_2, \ldots, e_m$ for some $m \geq 0$, where $(u_i, v_i) = e_i \in E$ for all $i$, $s = u_1$, $v_j = u_{j+1}$ for $1 \leq j < m$, and $v_m = t$.
We call $m$ the length of the walk.
For any vertex sets $S, T \subseteq V$, let $(S,T)$-walks refer to the set of walks $w$, where $w$ is a $(s,t)$-walk for some $s \in S$ and $t \in T$.

The adjacency matrix of $D$ is the $n$ by $n$ matrix such that the entry in row $i$ and column $j$ is the multiplicity of $(i,j)$ in $E$.
For a vertex set $S \subseteq V$, let $v_S$ be the characteristic vector for $S$: $v_S$ is a column vector where row $i$ is $1$ if $i \in S$ and $0$ otherwise.
If $A$ is the adjacency matrix of $D$, then the entry in row $i$ and column $j$ of $A^m$ is the number of $(i,j)$-walks of length $m$.
Moreover, $v_S^{T}A^mv_T$ is the number of walks of length $m$ in the set of $(S,T)$-walks.
The \emph{structure function,} $f_L(m) = v_I^{T}A^mv_F$, counts the number of files in a regular language $L$ of a given length.
When the language is clear, we will drop the subscript.

This paper is motivated by previous work in collaboration with Parker and Yancey \cite{Parker_Yancey_Yancey}, which developed several distance functions between regular languages.
The distance functions are based on the asymptotic behavior of $f_L(m)$ as $m$ grows, as described by Rothblum \cite{Rothblum_expansion_of_sums}.
The technical version of this statement is given in Theorem \ref{limiting behavior}.
Although Rothblum has many results on this topic, we will refer to this theorem as ``Rothblum's Theorem.'' 
See \cite{Rothblum_chapter} for a survey of work done by Rothblum and those who would come after.
The goal of the present work is to describe the asymptotics of the structure function from several perspectives, with an emphasis for intuitive results that are constructive in a practical way.
We provide a mix of new results and new proofs to known results.
For example, we give a simpler proof to Rothblum's Theorem (see Theorem \ref{limiting behavior}), and then prove that the asymptotic behavior is based on eigenvectors when it was previously known to rely on generalized eigenvectors (see Theorem \ref{general to regular}).

Chomskey and Miller \cite{Chomskey_Miller} originally claimed that ``Frobenius established... $f(\lambda) = a_1r_1^\lambda + a_2r_2^\lambda + \cdots + a_nr_n^\lambda$,'' where the $r_i$ are the eigenvalues of $A$ (we will be using $\lambda$ to denote eigenvalues everywhere besides this quote).
Moreover, the claim went on to state that there exists an $i$ such that $\|r_i\| > \max_{j \neq i} \|r_j\|$.
Unfortunately, Perron-Frobenius theory requires a set of assumptions that are not satisfied by general digraphs, including digraphs representing regular languages.
A more rigorous approach that compared the asymptotic growth of a regular language to a reference regular language was later developed \cite{Eilenberg,Salomma_Sittola} using generating functions.
Several other works have also applied the asymptotics of the structure function, where the asymptotic growth is sometimes described using generating functions \cite{Kozik,Bodrisky_Gartner_Oertsen_Schwinghammer} and sometimes described using ad-hoc methods \cite{Chang,Chan_Garofalakis_Rastogi,Cui_Dang_Fischer_Ibarra,Hansel_Perrin_Simon}.

Let us quickly recall the basics of Perron-Frobenius theory.
Let $M_{(s,t)}$ be the set of values $m$ such that there exists a $(s,t)$-walk of length $m$.
Perron's \cite{Perron} work used the assumptions that for all ordered pair of vertices $(s,t)$, (1) the set $M_{(s,t)}$ is non-empty and (2) the greatest common divisor among its elements is $1$.
Frobenius' \cite{Frobenius} work only used the first assumption.
A digraph is called \emph{irreducible} if it satisfies the first assumption, \emph{aperiodic} if it satisfies the second assumption, and \emph{primitive} if it satisfies both.
For an arbitrary matrix $M$, we can define an associated digraph $D_M$ where arc $uv \in E(D_M)$ if and only if the row $u$ column $v$ entry of $M$ is nonzero.
We then call $M$ irreducible/aperiodic/primitive if $D_M$ is irreducible/aperiodic/primitive.
A primitive matrix with nonnegative real elements does contain a unique largest eigenvalue, and so Chomskey and Miller's intuition is partially true in this restricted setting.
Moreover, the eigenvalue and each entry in the corresponding eigenvector are positive real numbers.
Similarly, an irreducible matrix with nonnegative real elements does contain a largest eigenvalue that is a positive real value and whose associated eigenvector contains only real \emph{nonnegative} entries.

In this paper, we examine the structure function from three different perspectives: enumerative combinatorics, linear algebra, and symbolic dynamics.
Using enumerative combinatorics, we re-examine the recursive sequence for the structure function established by Chomskey and Miller.
This is the content of Section \ref{enum comb section}.
This tact will allow us to quickly form an intuition for the important concepts.
This section will demonstrate that $f_L(m) = \sum_i \lambda_i^n p_i(m)$, where $p_i(x)$ is a polynomial with degree at most one less than the index of eigenvalue $\lambda_i$ of the adjacency matrix.

We consider linear algebra in Section \ref{lin alg section}.
In this paper, we use \emph{outerproduct} to refer to a matrix multiply of a column vector times a row vector into a single rank $1$ matrix---without conjugating the column vector as is sometimes done.
We will frequently make statements (for example, Theorem \ref{outerproduct} and Proposition \ref{orthogonal}) that look similar to a known fact, except without conjugation or with an inner product replaced with the matrix multiply of a row vector times a column vector.
Do not let this fool you: all of our statements from the perspective of linear algebra apply to the full generality of square matrices with complex entries.

The literature in linear algebra on the asymptotics of $A^m$ revolves around spectral projectors $E_{\lambda_i}$ (see \cite{Higham,Lindqvist}).
Eventually we will use techniques from dynamical systems to connect the asymptotic limit of $A^m$ to an outerproduct of eigenvectors.
Our first sequence of results is to converge the separate trains of thought by proving that the spectral projectors are indeed outerproducts of generalized eigenvectors.

\textbf{Theorem \ref{outerproduct}}
\textit{
Let $A$ be a matrix with eigenvalue $\lambda$ with algebraic multiplicity $m:=m(\lambda)$.
Let $v_{R,1}, \ldots, v_{R,m}$ and $v_{L,1}, \ldots, v_{L,m}$ denote a set of left and right generalized $\lambda$-eigenvectors such that each set is linearly independent.
Let $V_R$ denote the $n \times m$ matrix where column $i$ is $v_{R,i}$, and let $V_L$ denote the $m \times n$ matrix where row $i$ is $v_{L,i}$.
Then $V_LV_R$ is invertible and $E_\lambda = V_R (V_L V_R)^{-1} V_L$.
}

Our presentation of Theorem \ref{outerproduct} is simple and short.
We feel that this will clear up much of the mystery around spectral projectors, which has been a topic of their own interest.
To that end, we examine eight characterizations of a spectral projector from a larger survey by Agaev and Chebotarev \cite{Agaev_Chebotarev} and provide here a half-page proof of those eight (see Section \ref{spec proj as eig}) using Theorem \ref{outerproduct}.
We also relate spectral projectors to pseudo-inverses (Theorem \ref{drazin theorem} is a new proof to a statement in \cite{Meyer} and Corollary \ref{adjugate theorem} is new).

\textbf{Theorem \ref{drazin theorem} and Corollary \ref{adjugate theorem}}
\textit{
The Drazin inverse of $A$ is 
$$ A^D = \sum_{\lambda \neq 0} \sum_{i=0}^{\nu(\lambda)-1} \lambda ^{-1} (I - A \lambda^{-1})^{i}E_\lambda . $$
The adjugate of $A$ is 
$$ \adj(A) = \sum_{\lambda} \sum_{i=0}^{\nu(\lambda)-1} \left(\prod_{\lambda_* \neq \lambda}\lambda_*^{m(\lambda_*)} \right) \lambda^{m(\lambda) - 1 - i} (\lambda I - A)^{i}E_\lambda . $$
If $A$ is invertible, then
$$ A^{-1} = \sum_{\lambda} \sum_{i=0}^{\nu(\lambda)-1} \lambda ^{-1} (I - A \lambda^{-1})^{i}E_\lambda . $$
}

We continue to merge the results from symbolic dynamics and linear algebra by exploring when eigenvectors are sufficient without the heavier machinery of generalized eigenvectors.
It is already known that if a nonnegative matrix $A$ is primitive, then the limit of $A^m$ approaches $\rho^m P^*$, where $P^*$ is an outerproduct of $\rho$-eigenvectors for spectral radius $\rho$ (see \cite{Lind_Marcus,Rothblum_chapter}, equation 7.2.12 of \cite{Meyer}).
Our most surprising result is that eigenvectors are always sufficient.

\textbf{Corollary \ref{asymp growth for lin alg} and Theorem \ref{general to regular}}
\textit{
Let $A$ be a matrix with eigenvalues $\lambda _i$ and spectral radius $\rho$.
Let $S = \{i : \|\lambda_i\| = \rho\}$, $\nu = \max \{ \nu(\lambda_i) : i \in S \}$, and $S \supseteq T = \{i : \|\lambda_i\| = \rho, \nu(\lambda_i) = \nu\}$.
There exist matrices $\widehat{E_i}$ such that 
$$ \lim_{n \rightarrow \infty} \frac{  A^n } { {n \choose \nu - 1} \rho^{n - \nu + 1} }  - \sum_{i \in T} \left(\frac{\lambda_i}{\rho}\right)^{n - \nu + 1} \widehat{E_i}= 0. $$
For each $i$ there exist a basis of right $\lambda_i$-eigenvectors $v_{1,R}, \ldots, v_{t,R}$ and a set of left $\lambda_i$-eigenvectors $v_{1,L}, \ldots, v_{t,L}$ such that $\widehat{E_i} = \sum_{j=1}^t v_{j,R} v_{j,L}$.
If $A$ is a matrix with nonnegative real entries, then there exists a $q$ such that for each $k$, $ \lim_{n \rightarrow \infty} \frac{  A^{qn+k} } { {qn+k \choose \nu - 1} \rho^{qn + k - \nu + 1} } $ exists and converges to a sum of outerproducts of eigenvectors.
}

Finally, we consider results relating to symbolic dynamics in Section \ref{dynamics section}.
While not as general as the results from linear algebra, the results here will provide an intuitive explanation for the effects of sets $I$ and $F$ on the structure function.
In \cite{Parker_Yancey_Yancey}, we established that $A^m$ has similar asymptotic behavior to the structure function when $I$, $F$, and the digraph satisfy a condition called ``trimmed,'' but this is not true in general.

We consider edge shifts, which can be analyzed via the asymptotic behavior of $A^m$.
Previous results from symbolic dynamics only apply to irreducible digraphs.
Our approach for analyzing an edge shift is to construct a family of digraphs with corresponding edge shifts that (1) are easy to study, (2) have a homomorphism into the original edge shift, (3) the intersection of the images of the homomorphisms is asymptotically small when compared to the overall size of the edge shifts, and (4) the set of elements of the edge shift outside the union of the images of the homomorphisms is also asymptotically smaller.
We thus describe a method to break an edge shift into digestible pieces.


An \emph{irreducible component} of a digraph is a maximal sub-digraph that is irreducible; these are also known as \emph{strongly connected components}.
Let $D$ be a digraph with irreducible components $D_1, \ldots D_t$ whose respective adjacency matrices are $A$ and $A_1, \ldots, A_t$.
Let $p_i(x)$ be the characteristic function of $A_i$; the characteristic function of $A$ is then $\prod p_i(x)$.
This establishes a well-known relationship between the eigenvalues of $A$ and the eigenvalues of the $A_i$.
We will additionally require a relationship between the eigenvectors of $A$ and the eigenvectors of the $A_i$.
This lemma may be of independent interest to some (it generalizes a result of Rothblum \cite{Rothblum_eigenspaces} about dominant generalized eigenvectors of a matrix with nonnegative real entries), so we state it here.

\textbf{Lemma \ref{original to components}}
\textit{
Let $M$ be a general matrix with associated digraph $D$ that has irreducible components $D_1, \ldots D_t$.
Assume the irreducible components are ordered such that if the set of $(D_i,D_j)$-walks is non-empty, then $i < j$.
Let $M_1, \ldots, M_t$ be the submatrices of $M$ corresponding to $D_1, \ldots D_t$.
For a vector $v$, let $v^{(j)}$ be the sub-vector induced on $M_j$.
Let $v$ be a fixed generalized right $\lambda$-eigenvector of $A$, and let $i$ be the largest index such that $v^{(i)} \neq 0$.
Under these conditions, if $v$ has index $\nu$, then $v^{(i)}$ is a generalized right $\lambda$-eigenvector of $M_i$ with index at most $\nu$.
}


A dominant eigenvalue is an eigenvalue of matrix $M$ whose magnitude equals the spectral radius of $M$.
A dominant eigenvector is an eigenvector whose associated eigenvalue is dominant.
We have already established that the structure function is a sum of polynomials times an exponential function, and that the coefficients of these terms are based on the projection of the vectors $v_I, v_F$ onto the dominant eigenvectors of $A$.
We will consider the case when the dominant generalized eigenvectors have index $1$, and connect the projection of $v_I, v_F$ onto the dominant eigenvectors with the ``location'' of $I$ and $F$.


The assumption that the set of $(V(D_a),V(D_b))$-walks is empty is equivalent to assuming $\rho$ has index $1$, because of Rothblum's \cite{Rothblum_eigenspaces}  stronger statement:
if $D_{i_1}, D_{i_2}, \ldots, D_{i_r}$ is a maximum set of irreducible components whose spectral radius is equal to the spectral radius of $D$ and satisfy that for all $j < j'$ the set of $\left(V(D_{i_j}),V(D_{i_{j'}})\right)$-walks is nonempty, then the index of $\rho$ is $r$.
We say that vertex $v$ \emph{reaches} vertex set $S$ if the set of $(\{v\},S)$-walks is nonempty, and $v$ is \emph{reached from} $S$ if the set of $(S,\{v\})$-walks is nonempty.
For a vertex set $S$ and matrix $M$, we say that the \emph{$S$-mask of $M$} is a matrix $M'$ whose entries equal $0$ in rows and columns not in $S$ and equal the entries in $M$ otherwise.

\textbf{Theorem \ref{incomparable dominant}}
\textit{
Let $D$ be a digraph with irreducible components $D_1, \ldots D_t$ whose respective adjacency matrices are $A$ and $A_1, \ldots, A_t$.
Let $\rho$ be the spectral radius of $A$, $s$ the number of irreducible components with spectral radius $\rho$, and order the $D_i$ such that $\rho$ is the spectral radius of $A_1, \ldots, A_s$ and not $A_i$ for $i > s$.
Let $p_i$ be the period of $D_i$, and let $C_{i,1}, \ldots C_{i,p_i}$ be the periodic classes of $D_i$.
Set $P = \prod_i p_i$.
Let $V_{i,j}$ be the set of vertices reached or is reachable by $C_{i,j}$ in $A^{p_i}$.\\
Suppose for all $1 \leq a,b \leq s$ the set of $\left(V(D_a),V(D_b)\right)$-walks in $D$ is empty.
For $1 \leq i \leq s$ and $1 \leq j \leq p_i$, let $v_L^{\langle i,j \rangle}, v_R^{\langle i,j \rangle}$ be left, right $\rho^{p_i}$-eigenvectors of the $V_{i,j}$-mask of $A^{p_i}$, normalized such that $v_L^{\langle i,j \rangle} v_R^{\langle i,j \rangle} = 1$.
Under these conditions, for each row vector $w_L$, column vector $w_R$, and integer $k$ there exists $C, \epsilon > 0$ such that 
$$ w_L A^{Pm + k} w_R = \left(q(m) + \sum_{i=1}^s \sum_{j =1 }^{p_i} (w_L v_R^{\langle i,j \rangle})(v_L^{\langle i,j+k \rangle}w_R) \right) \rho^{Pm+k}, $$
where $q(m) < C ( 1 - \epsilon)^m$.
}

Theorem \ref{incomparable dominant} may appear similar to Corollary \ref{asymp growth for lin alg} and Theorem \ref{general to regular}.
However, there is an important distinction: by Perron-Frobenius theory the eigenvectors in Theorem \ref{incomparable dominant} are known to have all nonnegative entries. 
Hence Theorem \ref{incomparable dominant} establishes the coefficients of the structure function as a sum of nonnegative numbers based on the location of sets $I$ and $F$.
Moreover, the coefficients will be nonzero (and hence the structure function will have the same asymptotics as $A^m$) if and only if $I$ and $F$ intersect coordinates with non-zero entries in the associated eigenvectors.
This justifies our comments about the ``location'' of sets $I$ and $F$, which we will continue to elaborate on.

Theorem \ref{incomparable dominant} is quite practical, as it applies to many regular languages.
For example, it applies to any regular language with an accepting state.
As a second example, if a regular language $L$ over alphabet $\Sigma$ is such that the spectral radius of the digraph representing $L$ is $|\Sigma|$, then the digraph satisfies the assumptions of Theorem \ref{incomparable dominant}.
The compliment of a regular language is a regular language, and it is known \cite{Parker_Yancey_Yancey} that a regular language or its compliment will correspond to a digraph with spectral radius $|\Sigma|$.
Hence, Theorem \ref{incomparable dominant} applies to at least ``half'' of all regular languages.
Theorem \ref{incomparable dominant} also applies to Markov chains and stochastic matrices.

We now establish that the ``right location'' for $I$ ($F$) is before (after) the relative irreducible components.
Rothblum \cite{Rothblum_eigenspaces} proved this result, but only for nonnegative real matrices and eigenvectors that are dominant.

\textbf{Corollary \ref{up set, down set}}
\textit{
Let $M$ be a general matrix with associated digraph $D$ that has irreducible components $D_1, \ldots D_t$.
Let $M_1, \ldots, M_t$ be the submatrices of $M$ corresponding to $D_1, \ldots D_t$.
Let $v_R$ and $v_L$ be right and left generalized $\lambda$-eigenvectors of $M$.
\begin{itemize}
	\item If $v_R$ is nonzero in coordinate $u$, then there exists a path from $u$ to $w$, where $w \in D_i$, $\lambda$ is an eigenvalue of $M_i$, and $v_R$ is nonzero at $w$.
	\item If $v_L$ is nonzero in coordinate $u$, then there exists a path from $w$ to $u$, where $w \in D_j$, $\lambda$ is an eigenvalue of $M_j$, and $v_L$ is nonzero at $w$.
\end{itemize}
}

The paper is broken up as follows.
In Section 2 we illustrate the above theorems with three examples.
In Section 3 we cover the enumerative combinatorics approach.
The advantages of this section are its simplicity and brevity.
In Section 4 we cover the linear algebra approach.
The advantage of this section is its generality.
In Section 5 we cover the symbolic dynamics approach.
The advantage of this section is the intuitive description it provides for the effects of sets $I$ and $F$.
The full purpose of this paper is how the different sections interact.
For example, Corollary \ref{up set, down set} enhances the beauty of Corollary \ref{asymp growth for lin alg} and Theorem \ref{general to regular} as much as it does Theorem \ref{incomparable dominant}.

\section{Examples}\label{ex section}
We will analyze the regular languages corresponding to the regular expressions $a^*ba^*b(a|b)^*$, $(a|b)^*\left( \left(d(e|f)^*\right)|\left(c(d(e|f|g|h)\right)^*\right)$, and \\
$\left(a(b|c)^*\right) | \left( \left( (b|c)(a|b) \right)^*(b|c|d) \right)$.
The digraph Chomskey and Miller use to analyze a regular language is known as a Deterministic Finite Automata (DFA).
The DFA for the three examples are illustrated in Figure \ref{example figure}.
The set of vertices in $I$ are known as the initial states, and the set $F$ is known as the set of final states.
All three examples are trimmed DFA, and hence $A^n$ will have the same asymptotic behavior as $f(n)$.
However, the first two examples easily extend to the general case.

\begin{center}
\begin{figure}
\includegraphics[width=14cm]{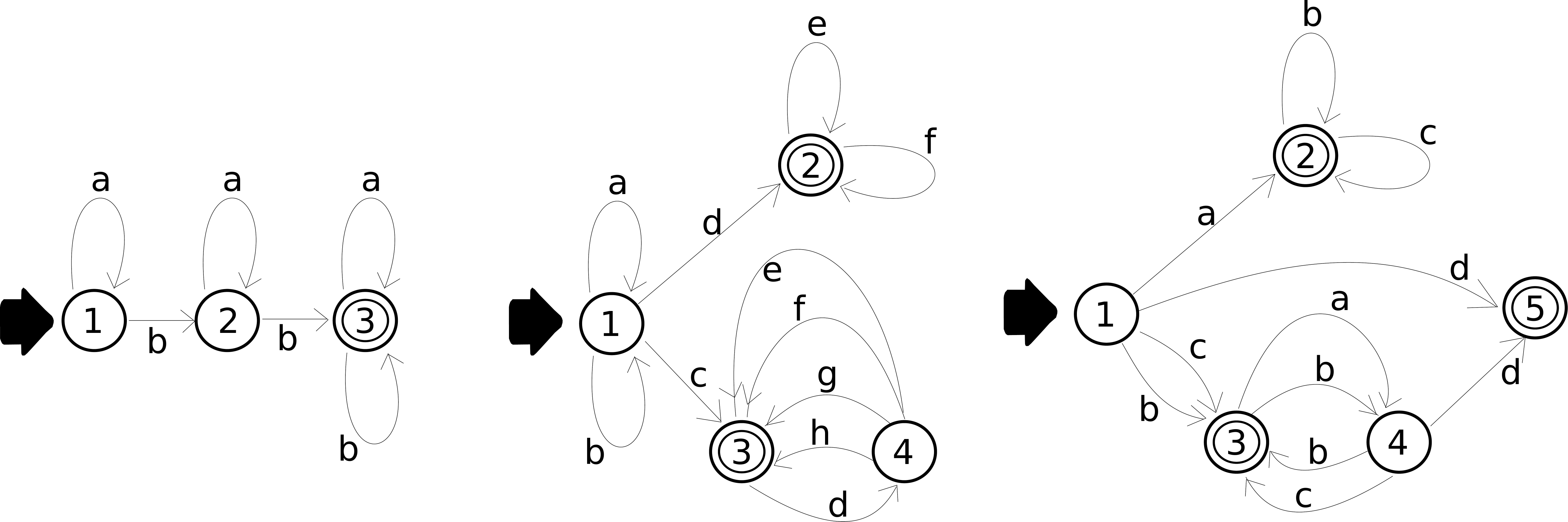}
\caption{Illustrations of the DFA for the three examples in Section \ref{ex section}.  They are DFA representing the regular expressions $a^*ba^*b(a|b)^*$, $(a|b)^*\left( \left(d(e|f)^*\right)|\left(c(d(e|f|g|h)\right)^*\right)$, and 
$\left(a(b|c)^*\right) | \left( \left( (b|c)(a|b) \right)^*(b|c|d) \right)$, in order, from left to right.
The thick arrow indicates the unique initial state for each DFA, and double circles signify a final state.
The edge labels are used to construct the DFA from a regular expression and otherwise are irrelevant.
}
\label{example figure}
\end{figure}
\end{center}

\textbf{Example 1.} Consider the regular language defined by the regular expression $a^*ba^*b(a|b)^*$.
It is all words over the alphabet $\{a,b\}$ such that $b$ appears at least twice.
It should be clear that $f(n) = 2^n - n - 1$.
We will recreate this formula using arguments from Section \ref{enum comb section}.
This DFA has initial vector, adjacency matrix, and final vector as follows
$$ \left( \begin{array}{c} 1 \\ 0 \\ 0 \\ \end{array} \right), \left( \begin{array}{ccc} 1& 1 & 0 \\ 0 & 1 & 1 \\ 0 & 0 & 2 \\ \end{array} \right), \left( \begin{array}{c} 0 \\ 0 \\ 1 \\ \end{array} \right). $$
Because $(A - 2 I) (A - I)^2 = 0$, we get the recurrence relation 
$$ f(n) - 4 f(n-1) + 5f(n-2) - 2 f(n-3) = 0, n \geq 3,$$
and we see that $f(n) = (a n + b) 1^n + c 2^n$ for coefficients $a,b,c$.
Using initial conditions $f(0) = v_I ^T v_F = 0$, $f(1) = v_I ^T A v_F = 0$, $f(2) = v_I ^T A^2 v_F = 1$, we can solve to get $c = 1$, $a = b = -1$, which matches our original formula.

\textbf{Example 2.} Consider the regular language defined by the regular expression 
$$(a|b)^*\left( \left(d(e|f)^*\right)|\left(c(d(e|f|g|h)\right)^*\right).$$
This regular expression can be represented by a DFA with initial vector, adjacency matrix, and final vector as follows
$$\left( \begin{array}{c} 1 \\ 0 \\ 0 \\ 0 \\ \end{array} \right), \left( \begin{array}{c|c|cc} 2 & 1 & 1 & 0 \\\hline 0 & 2 & 0 & 0 \\ \hline 0 & 0 & 0 & 1 \\ 0 & 0 & 4 & 0 \\ \end{array} \right), \left( \begin{array}{c} 0 \\ 1 \\ 1 \\0 \\ \end{array} \right)  .$$
The matrix has three irreducible components, which correspond to vertices $\{1\}, \{2\}, \{3,4\}$.
The lines drawn in matrix $A$ partition it into $9$ sub-matrices according to the irreducible components.
Let $B_{i,j}$ be the submatrix that is $i$ from the top and $j$ from the right (so $B_{1,3} = (1, 0)$ and $B_{3,2} = (0, 0)^T$).
The irreducible components are ordered such that the matrix is in Frobenius Normal Form, which is when $B_{i,j} = 0$ for $i > j$.

$A$ has $2$ as an eigenvalue with algebraic multiplicity $3$ and $-2$ as an eigenvalue with algebraic multiplicity $1$.
The geometric multiplicity of $2$ is $1$, and a $2$-eigenvector is $(1,0,0,0)^T$.
Two generalized $2$-eigenvectors of $A$ are $(100, 1, 0, 0)^T$ and $(100, 0, 1, 2)^T$, each with index $2$, and these three vectors form a basis for the $2$-eigenspace.
Lemma \ref{original to components} claims that $(1)^T$ is a $2$-eigenvector of $B_{1,1}$, that $(1)^T$ is a generalized $2$-eigenvector of $B_{2,2}$ with index at most $2$, and that $(1,2)^T$ is a generalized eigenvector of $B_{3,3}$ with index at most $2$.
Left $2$-eigenvectors of $A$ include $(0,1,0,0)$ and $(0,0,2,1)$, and $(4,100,1,0)$ is a left generalized $2$-eigenvector with index $2$.
Theorem \ref{outerproduct} states that the spectral projector for eigenvalue $2$ is then 
$$ E_2 = \left( \begin{array}{ccc} 1 & 100 & 100 \\ 0 & 1 & 0 \\ 0 & 0 & 1 \\ 0 & 0 & 2 \\ \end{array} \right)
   \left(   \left( \begin{array}{cccc} 0 & 1 & 0 & 0 \\ 0 & 0 & 2 & 1 \\ 4 & 100 & 1 & 0 \\ \end{array} \right)  \left( \begin{array}{ccc} 1 & 100 & 100 \\ 0 & 1 & 0 \\ 0 & 0 & 1 \\ 0 & 0 & 2 \\ \end{array} \right) \right)^{-1}
   \left( \begin{array}{cccc} 0 & 1 & 0 & 0 \\ 0 & 0 & 2 & 1 \\ 4 & 100 & 1 & 0 \\ \end{array} \right) $$
$$  = \left( \begin{array}{cccc} 1 & 0 & 1/8 & -1/16 \\ 0 & 1 & 0 & 0 \\ 0 & 0 & 1/2 & 1/4 \\ 0 & 0 & 1 & 1/2 \\ \end{array} \right). $$
Similarly the spectral projector for eigenvalue $-2$ is 
$$ E_{-2} =  \left( \begin{array}{cccc} 0 & 0 & -1/8 & 1/16 \\ 0 & 0 & 0 & 0 \\ 0 & 0 & 1/2 & -1/4 \\ 0 & 0 & -1 & 1/2 \\ \end{array} \right). $$
We leave it to the reader to confirm the various properties of Theorem \ref{define spectral projector} hold, such as $E_j ^2 = E_j, E_2 E_{-2} = E_{-2}E_2 = 0$, $A E_j = E_j A$, $E_2 + E_{-2} = I$.
By Theorem \ref{limiting behavior}, we have that 
\begin{eqnarray*}
 A^n  	& = & \sum_\lambda \sum_{i=0}^{\nu(\lambda)-1} {n \choose i} \lambda^{n-i} (A - \lambda I)^{i} E_\lambda \\
	& = &  {n \choose 0} 2^n E_{2} + {n \choose 1}2^{n-1}(A - 2 I)E_2 + {n \choose 0} (-2)^n E_{-2} \\
	& = & 2^n\left(  \left( \begin{array}{cccc} 0 & 1/2 & 1/4 & 1/8 \\ 0 & 0 & 0 & 0 \\ 0 & 0 & 0 & 0 \\ 0 & 0 & 0 & 0 \\ \end{array} \right)n +  \left( \begin{array}{cccc} 1 & 0 & 1/8 & -1/16 \\ 0 & 1 & 0 & 0 \\ 0 & 0 & 1/2 & 1/4 \\ 0 & 0 & 1 & 1/2 \\ \end{array} \right) \right) \\
	& & + (-2)^n \left( \begin{array}{cccc} 0 & 0 & -1/8 & 1/16 \\ 0 & 0 & 0 & 0 \\ 0 & 0 & 1/2 & -1/4 \\ 0 & 0 & -1 & 1/2 \\ \end{array} \right).
\end{eqnarray*}
This confirms that $A^n = \sum_\lambda \lambda^n p_\lambda(n)$, where $p_2(x) = \frac12(A - 2 I)E_2 x  + E_2$ and $p_{-2}(x) = E_{-2}$.
For $i \in \{0,1\}$ we have that $A^{2n+i} = 2^{2n+i} S_i(2n+i)$, where $S_0(x) = \frac12(A - 2 I)E_2 x  + E_2 + E_{-2}$ and $S_1(x) = \frac12(A - 2 I)E_2 x  + E_2 - E_{-2}$.
By Corollary \ref{asymp growth for lin alg} and Theorem \ref{general to regular}, we see that 
$$ \lim_{n \rightarrow \infty} \frac{A^n}{{n \choose 1} 2^{n-1}} = \widehat{E_2} = (A - 2 I)E_2 = \left( \begin{array}{cccc} 0 & 1 & 1/2 & 1/4 \\ 0 & 0 & 0 & 0 \\ 0 & 0 & 0 & 0 \\ 0 & 0 & 0 & 0 \\ \end{array} \right) $$
$$ = \left( \begin{array}{c} 1 \\ 0 \\ 0 \\ 0 \\ \end{array} \right) (0, 1, 1/2, 1/4), $$
where $(1,0,0,0)^T$ is a $2$-eigenvector and $(0,1,1/2,1/4)$ is a left $2$-eigenvector ($-2$-eigenvectors are not used because $\nu(2) = 2 > 1 = \nu(-2)$).

\textbf{Example 3.} Consider the regular language defined by the regular expression 
$$\left(a(b|c)^*\right) | \left( \left( (b|c)(a|b) \right)^*(b|c|d) \right).$$
This regular expression can be represented by a DFA with initial vector, adjacency matrix, and final vector as follows
$$\left( \begin{array}{c} 1 \\ 0 \\ 0 \\ 0 \\ 0\\ \end{array} \right), \left( \begin{array}{c|c|cc|c} 0 & 1 & 2 & 0 & 1 \\\hline 0 & 2 & 0 & 0 & 0 \\ \hline 0 & 0 & 0 & 2 & 0 \\ 0 & 0 & 2 & 0 & 1 \\ \hline 0 &0 &0 &0 &0 \\  \end{array} \right), \left( \begin{array}{c} 0 \\ 1 \\ 1 \\ 0 \\ 1 \\ \end{array} \right)  .$$
The matrix has four irreducible components, which correspond to vertices $\{1\}, \{2\}, \{3,4\}, \{5\}$.
The lines drawn in matrix $A$ partition it into $16$ sub-matrices according to the irreducible components.
Let $B_{i,j}$ be defined as before, and let $A_i = B_{i,i}$.

The spectral radius of $A$ is $2$, and $2$ is the spectral radius of $A_2, A_3$ but not $A_1, A_4$.
There are no ($\{3,4\}$, $\{2\}$)-walks or $(\{2\}, \{3,4\})$-walks, so the assumptions of Theorem \ref{incomparable dominant} are satisfied.
We have $p_2 = 1$ and $p_3 = 2$.
The only vertex reached by $\{2\}$ is $\{2\}$, and the vertices that reach $\{2\}$ are $\{1,2\}$.
So $V_{2,1} = \{1,2\}$, and thus to calculate $v_L^{\langle2,1\rangle}, v_R^{\langle2,1\rangle}$ we need to consider the $\{1,2\}$-mask of $A^1$, which is $ \left( \begin{array}{ccccc} 0 & 1 & 0 & 0 & 0 \\ 0 & 2 & 0 & 0 & 0 \\ 0 & 0 & 0 & 0 & 0 \\ 0 & 0 & 0 & 0 & 0 \\ 0 &0 &0 &0 &0 \\  \end{array} \right) $.
The unique dominant eigenvectors (after appropriate scaling) of this matrix are $v_L^{\langle2,1\rangle} = (0, 1,0,0,0), v_R^{\langle2,1\rangle} = ( 1/2,  1, 0, 0, 0)^T$.

\begin{center}
\begin{figure}
\begin{center}
\includegraphics[width=5cm]{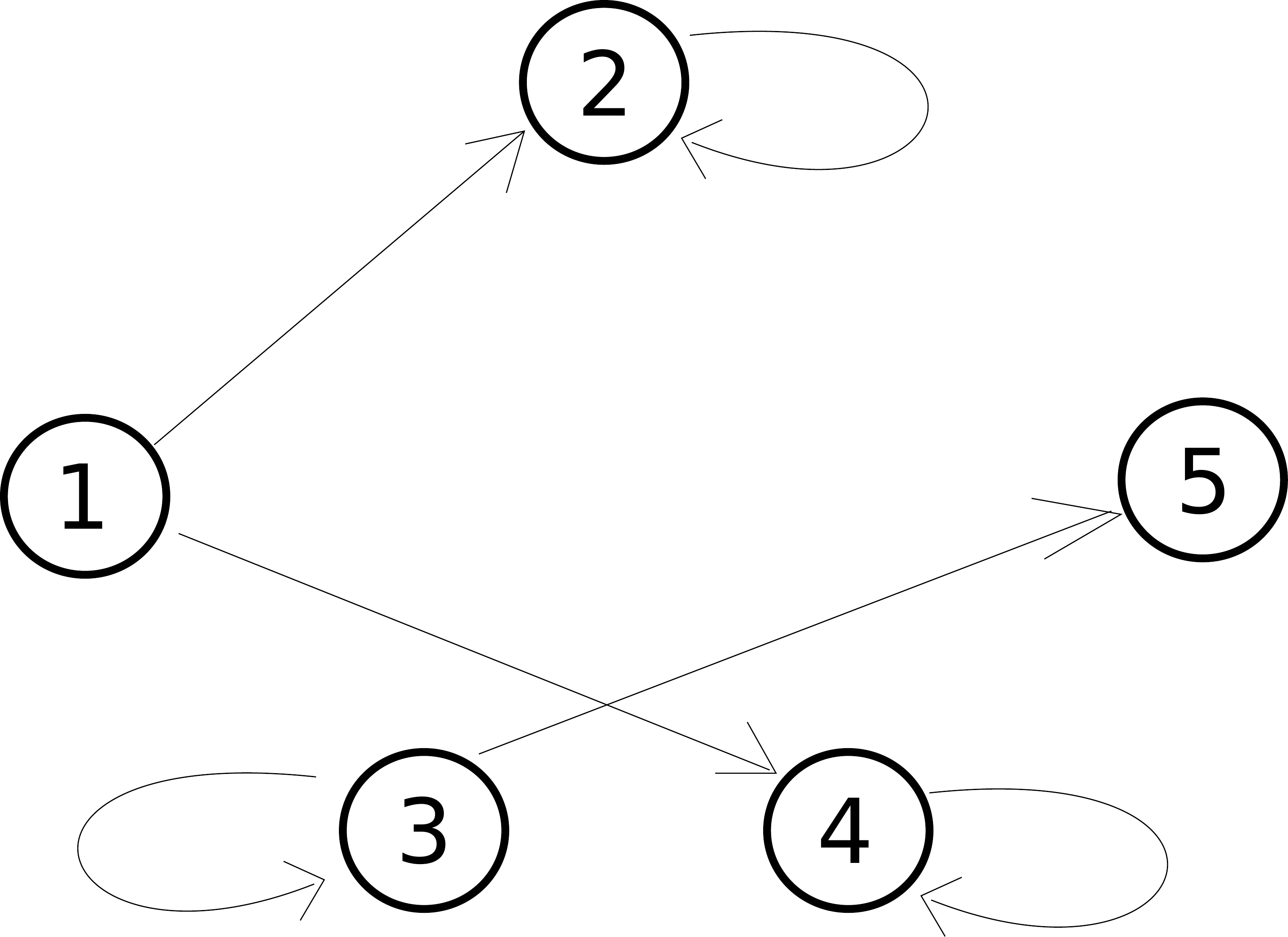}

\caption{For any matrix, we can construct an associated digraph where directed edge $ij$ exists only if the entry in row $i$ column $j$ is nonzero.
Using this associated digraph, we can generalize terms such as ``irreducible component'' and ``reached by'' from digraphs to matrices.
The above image is the digraph associated to $A^2$ in the third example of Section \ref{ex section}.}
\label{example 3 figure}
\end{center}
\end{figure}
\end{center}

The periodic classes of the irreducible component $\{3,4\}$ in $A$ are $\{3\}$ and $\{4\}$.
To calculate $v_L^{\langle3,1\rangle}, v_R^{\langle3,1\rangle}, v_L^{\langle3,2\rangle}, v_R^{\langle3,2\rangle}$, we consider 
$$ A^2 = \left( \begin{array}{ccccc} 0 & 2 & 0 & 4 & 0 \\ 0 & 4 & 0 & 0 & 0 \\  0 & 0 & 4 & 0 & 2 \\ 0 & 0 & 0 & 4 & 0 \\  0 &0 &0 &0 &0 \\  \end{array} \right) .$$
The digraph associated to $A^2$ is illustrated in Figure \ref{example 3 figure}.
The vertices reaching $\{3\}$ in $A^2$ are only $\{3\}$, and the vertices reached by $\{3\}$ in $A^2$ are $\{3, 5\}$.
Thus $V_{3,1} = \{3,5\}$, and the $V_{3,1}$-mask of $A^2$ is then $\left( \begin{array}{ccccc} 0 & 0 & 0 & 0 & 0 \\ 0 & 0 & 0 & 0 & 0 \\  0 & 0 & 4 & 0 & 2 \\ 0 & 0 & 0 & 0 & 0 \\  0 &0 &0 &0 &0 \\  \end{array} \right)$.
It follows that $v_L^{\langle3,1\rangle} = (0, 0,1,0,1/2), v_R^{\langle3,1\rangle} = ( 0,  0, 1, 0, 0)^T$.
The vertices reaching $\{4\}$ in $A^2$ are $\{1,4\}$, and the vertices reached by $\{4\}$ in $A^2$ are $\{4\}$.
Therefore $V_{3,2} = \{1,4\}$, and the $V_{3,2}$-mask of $A^2$ is $\left( \begin{array}{ccccc} 0 & 0 & 0 & 4 & 0 \\ 0 & 0 & 0 & 0 & 0 \\  0 & 0 & 0 & 0 & 0 \\ 0 & 0 & 0 & 4 & 0 \\  0 &0 &0 &0 &0 \\  \end{array} \right) .$
It follows that $v_L^{\langle3,2\rangle} = (0, 0,0,1,0), v_R^{\langle3,2\rangle} = ( 1,  0, 0,  1, 0)^T$.

We may now apply Theorem \ref{incomparable dominant} to say that for large positive integer $n$, 
\begin{eqnarray*}
 f(2n)		& \approx	& 2^{2n} \left( (v_I^T v_R^{\langle2,1\rangle})(v_L^{\langle2,1\rangle} v_F) + (v_I^T v_R^{\langle3,1\rangle})(v_L^{\langle3,1\rangle} v_F) + (v_I^T v_R^{\langle3,2\rangle})(v_L^{\langle3,2\rangle} v_F) \right) \\
		& = 		& 2^{2n} ( \frac{1}{2} \times 1 + 0 \times \frac{3}{2} + 1 \times 0 ) \\
		& = 		& 2^{2n-1}, \\
 f(2n+1)	& \approx	& 2^{2n+1}  \left( (v_I^T v_R^{\langle2,1\rangle})(v_L^{\langle2,1\rangle} v_F) + (v_I^T v_R^{\langle3,1\rangle})(v_L^{\langle3,2\rangle} v_F) + (v_I^T v_R^{\langle3,2\rangle})(v_L^{\langle3,1\rangle} v_F) \right) \\
		& = 		& 2^{2n+1} ( \frac{1}{2} \times 1 + 0 \times 0 + 1 \times \frac{3}{2} ) \\
		& =		& 2^{2n+2}.
\end{eqnarray*}
Importantly, the vectors $v_L^{\langle i, j\rangle}, v_R^{\langle i, j\rangle}$ have all nonnegative real entries, and therefore the contribution of the entries in $I$ and $F$ is intuitive.
Our results are consistent with Corollary \ref{up set, down set}: for the asymptotic behavior of $A^n$ to match that of $f(n) = v_I^T A^n v_F$ the coefficients we calculate above must be nonzero, which happens only when $F$ is reached by the irreducible components $A_2, A_3$ with largest spectral radius and $I$ reaches them (which is a weaker condition than being trimmed).
\section{Enumerative Combinatorics} \label{enum comb section}

Let $L$ be a regular language with structure function $f := f_L$ and associated digraph $D$ whose adjacency matrix is $A$.
Recall that there exist vectors $v_I$ and $v_F$ such that $f(m) = v_I^T A^m v_F$.
Let $p_A(x) = \left| A - xI \right|$ be the characteristic polynomial for $A$.
Chomskey and Miller \cite{Chomskey_Miller} noted that the Cayley-Hamilton theorem states that $p_A(A) = 0$, and therefore $p_A$ describes a recursive characterization for $f$.

It is also well known that there exists a minimum polynomial $m_A$ such that $m_A(A) = 0$, which satisfies $m_A | p_A$.
If $m_A(x) = \prod_i (\lambda_i - x)^{\nu(\lambda_i)}$, where $\lambda_i \neq \lambda_j$ when $i \neq j$, then $\nu(\lambda_i)$ is called the \emph{index} of $\lambda_i$.
Let $N = \sum_i \nu(\lambda_i)$, $m_A(x) = \sum_{j=0}^N a_j x^j$, and $a_N = 1$.
When $m \geq N$ we have that 
\begin{eqnarray*}
f(m)			& = & v_{I}^T A^m v_F \\
			& = & v_{I}^T \left(- \sum_{i=0}^{N-1} a_i A^{m-N+i}\right)  1_F \\
			& = & \sum_{i=0}^{n-1} -a_i \left(v_{I}^T A^{m-N+i} v_F\right) \\
			& = & \sum_{i=0}^{n-1} -a_i f(m-N+i).
\end{eqnarray*}

So the values of $f(m)$ are defined by a linear homogeneous recurrence relation with constant coefficients whose characteristic polynomial is $m_A$.
Therefore $f(m) = \sum_i \lambda_i^n p_i(m),$ where $p_i(x)$ is a polynomial with degree at most $\nu(\lambda_i)-1$.
The coefficients of $p_i$ depend on $v_I, v_F$; sometimes $p_i(x) = 0$.
To better understand such values, we turn to other fields of mathematics.

Before we turn to a different subject area, let us first establish how this result implies one of Rothblum's foundational theorems.
The recursive formula holds for any pair of vectors $v_I, v_F$, so by setting $I=\{a\}$ and $F = \{b\}$ we create a formula for the entry in row $a$ and column $b$ of $A^m$.
For fixed $a,b$, let the coefficients of $p_i(x)$ be denoted as $c_{((a,b),i,j)}$ such that $p_i(x) = \sum_{j=0}^{\nu(\lambda_i)-1} c_{((a,b),i,j)} x^j$.

Order the eigenvalues of $A$ such that $\|\lambda_1\| \geq \|\lambda_2\| \geq \cdots$, let $\rho = \|\lambda_1\|$, and let $s$ be such that $\|\lambda_1\| =  \cdots = \|\lambda_s\| > \|\lambda_{s+1}\|$.
It is well known that the set of eigenvalues of $A$ is the union of the set of eigenvalues for the adjacency matrix of each irreducible component.
Perron \cite{Perron} showed that for each irreducible digraph $D_i$ with adjacency matrix $A_i$ there exists an integer $p_i$ such that if $\lambda'$ is an eigenvalue of $A_i$ such that $\|\lambda'\|\geq\rho$, then $\lambda' = \rho e^{2\pi j/p_i}$ for some integer $j$.
By taking $P$ to be the least common multiple of the $p_i$, we see that for $1 \leq i \leq s$ we have that $\lambda_i^P = \rho^P$.
We have thus given a short proof to Rothblum's theorem \cite{Rothblum_expansion_of_sums} that for nonnegative real matrix $A$ with spectral radius $\rho$, there exists an integer $P$ and matrix polynomials $S_0(x), \ldots, S_{P-1}(x)$ such that $\lim_{m\rightarrow \infty} (A/\rho)^{Pm+k} - S_k(Pm+k) = 0$.
Moreover, we can additionally show that if $S_k(x) = \sum_{j} M_j x^j$ for matrices $M_j$, then the entry in row $a$ column $b$ of $M_j$ is $\sum_{i=1}^s \left(\frac{\lambda_i}{\rho}\right)^k c_{((a,b),i,j)}$.
A second proof of Rothblum's result will appear in Section \ref{spec proj poly}.

\section{Linear Algebra} \label{lin alg section}
\subsection{Background}

In this Section we assume all matrices are square.
A generalized (right) $\lambda$-eigenvector of a matrix $A$ is a vector $v$ such that $(A - \lambda I)^k v = 0$ for some $k$.
A generalized left $\lambda$-eigenvector is a row vector $v$ such that $v (A - \lambda I)^k  = 0$ for some $k$.
The minimum $k$ such this is true is called the \emph{index} of $v$ and is denoted by $\nu(v)$.
As an abuse of notation, let $\nu(\lambda)$ denote the maximum index among generalized $\lambda$-eigenvectors.
If $w = (A - \lambda I)^{t} v$ for $t < \nu(v)$, then $w$ is a generalized $\lambda$-eigenvector with index $\nu(v) - t$.
The set of eigenvectors are the set of generalized eigenvectors with index $1$.

If $p_A(x)$ is the characteristic function for $A$, and $m(\lambda)$ is the multiplicity of $\lambda$ as a root of $p_A(x)$, then the generalized $\lambda$-eigenvectors form a space whose dimension is $m(\lambda)$.
Moreover, if $A$ is a square matrix with $n$ rows, then there exists a basis of $\mathbb{C}^n$ using generalized eigenvectors of $A$.
Let $m_A(x)$ be the minimal polynomial for $A$; we define this to be $m_A(x) = \prod_\lambda(x - \lambda)^{\nu(\lambda)}$.
We have that $\nu(\lambda) \leq m(\lambda)$, so $m_A(x) | p_A(x)$.

We first give two statements that will be useful later.

\begin{claim} \label{everything is an outerproduct}
Let $A$ be a matrix whose row space is spanned by linearly independent row vectors $u_1, u_2, \ldots, u_k$ and whose column space is spanned by column vectors $v_1, v_2, \ldots, v_k$.
Then there exists column vectors $v_1', v_2', \ldots, v_k'$ such that $A = \sum_{i=1}^k v_i' u_i$.
\end{claim}
\begin{proof}
Let $a_i$ denote the row vector that is $1$ in coordinate $i$ and $0$ in all other coordinates.
Because the $u_1, u_2, \ldots, u_k$ are linearly independent, there exists an invertible linear transformation $Q$ such that $u_i Q = a_i$.
Consider the matrix $A Q$; its rows space is spanned by $a_i$ for $1 \leq i \leq k$, and its column space is spanned by $v_1, v_2, \ldots, v_k$.
Let $v_i'$ denote column $i$ in the matrix $A Q$, so that $AQ = \sum_i v_i' a_i$.
The claim then follows from $A = AQQ^{-1} = \left(\sum_i v_i' a_i\right) Q^{-1} = \sum_i v_i' a_i Q^{-1} = \sum_i v_i' u_i$.
\end{proof}

For completeness, we include a proof of the following statement.

\begin{theorem}[Theorem VII.1.3 of \cite{Dunford_Schwartz}]
Let $p_1(x)$ and $p_2(x)$ be polynomials.
We have that $p(A) = q(A)$ if and only if $m_A(x) | (p_1(x) - p_2(x))$. 
\end{theorem}
\begin{proof}
Without loss of generality, assume that $p_2 \cong 0$.
We wish to show that $p_1(A) = 0$ if and only if $m_A(x) | p_1(x)$.

Clearly a matrix $M$ equals $0$ if and only if $Mv = 0$ for all vectors $v$.
Let $v_1, \ldots, v_n$ be generalized eigenvectors of $A$ that form a basis.
Because $A$ and $I$ commute with themselves, we can re-arrange the terms of $m_A(A)$ and $p_1(A)$.
So if $v_i$ is a generalized $\lambda$-eigenvector, then 
$$m_A(A)v_i = \left(\prod_{\lambda' \neq \lambda}(A - I \lambda')^{\nu(\lambda')}\right) (A - I \lambda)^{\nu(\lambda)}v_i = 0.$$ 
Since $m_A(A)$ sends each element of a basis to $0$, it equals $0$.
Therefore if  $m_A(x) | p_1(x)$, then $p_1(A) = 0$.

Now suppose that $\lambda$ is a root of $p_1(x)$ with multiplicity $k < \nu(\lambda)$.
Let $v'$ be a generalized $\lambda$-eigenvector with index $k+1$, and let $v = (A - \lambda I)^k v'$.
By above, $v$ is a $\lambda$-eigenvector.
Therefore 
$$ p_1(A) v' = \left(\prod_{c \neq \lambda} (A - c I)^{k_c} \right) (A - \lambda I)^k v' = \prod_{c \neq \lambda} (A - c I)^{k_c}v = \prod_{c \neq \lambda} (\lambda - c)^{k_c} v \neq 0.$$
Therefore if $p_1(A) = 0$, then $m_A(x) | p_1(x)$.
\end{proof}

\begin{corollary} \label{zeroes to derrivatives}
Let $p_1(x)$ and $p_2(x)$ be two polynomials, and let $A$ be a matrix.
If for each eigenvalue $\lambda$ of $A$ we have that $\frac{d^i}{dx^i}p_1(\lambda) = \frac{d^i}{dx^i} p_2(\lambda)$ for all $0 \leq i \leq \nu(\lambda)-1$, then $p_1(A) = p_2(A)$.
\end{corollary}

\subsection{Spectral Projectors as Polynomials} \label{spec proj poly}

The following is the definition behind \emph{spectral projectors} by Dunford and Schwartz \cite{Dunford_Schwartz}.
We will use it here for its simplicity.
An intuitive description of a spectral projector will be presented soon.

\begin{theorem}[folklore] \label{define spectral projector}
Let $\lambda_i$ be the eigenvalues of a matrix $A$.
There exists matrices $E_{\lambda_i}$ such that \\
(1) $E_{\lambda_i} ^2 = E_{\lambda_i}$ \\
(2) $E_{\lambda_i} E_{\lambda_j} = 0$ when $i \neq j$, \\
(3) $\sum_i E_{\lambda_i} = I$, and \\
(4) $A E_{\lambda_i} = E_{\lambda_i} A$.
\end{theorem}
\begin{proof}
Let $e_i(x)$ be a polynomial such that $\frac{d^j}{dx^j}e_i(\lambda_k)$ equals $1$ if $j=0$ and $k=i$ and equals zero in all other cases when $0 \leq j < \nu(\lambda_k)$.
Let $E_{\lambda_i} = e_i(A)$.
Because $E_{\lambda_i}$ is a polynomial of $A$, (4) clearly holds.
To see why the rest of the proof is true, apply Corollary \ref{zeroes to derrivatives} for\\
(1) $p_1(x) = e_i(x)^2$ and $p_2(x) = e_i(x)$, \\
(2) $p_1(x) = e_i(x)e_j(x)$ and $p_2(x) = 0$, and \\
(3) $p_1(x) = \sum_i e_i(x)$ and $p_2(x) = 1$.
\end{proof}

The $E_{\lambda_i}$ are sometimes called the \emph{components}, because the behavior of $A$ can be split into a sum of behaviors on the $E_{\lambda_i}$.
From part (3) of Lemma \ref{define spectral projector}, we see that
\begin{equation}\label{this is what A is}
A = A \left( \sum_i E_{\lambda_i} \right) = \sum_i A E_{\lambda_i} = \sum_i \left( \lambda_i E_{\lambda_i} + (A - \lambda_i I) E_{\lambda_i} \right).
\end{equation}
Now using parts (1), (2), and (4) of Lemma \ref{define spectral projector}, we have that 
\begin{eqnarray*}
A^n		& = & 		\left( \sum_i \left( \lambda_i E_{\lambda_i} + (A - \lambda_i I) E_{\lambda_i} \right) \right)^n     	\\
		& = &		\sum_i \left( \lambda_i E_{\lambda_i} + (A - \lambda_i I) E_{\lambda_i} \right)^n			\\
		& = &		\sum_i		\sum_{j=0}^n	{n \choose j} \lambda_i ^{n-j} (A - \lambda_i I)^{j} E_{\lambda_i}	.
\end{eqnarray*}

It is well-known that the space spanned by the columns of $E_0$ is the null space of $A^{\nu(0)}$ (for example, see Theorem VII.1.7 of \cite{Dunford_Schwartz}).
The generalized version of this statement is that $(A - \lambda I)^{\nu(\lambda)}E_\lambda = 0$.
This can be seen to be true from applying Corollary \ref{zeroes to derrivatives} with $p_1(x) = (x - \lambda)^{\nu(\lambda)}e(x)$ and $p_2(x) = 0$.
Thus the majority of our summation above can be ignored.
This gives us the following theorem.

\begin{theorem} \label{limiting behavior}
Let $A$ be a matrix with eigenvalues $\lambda_i$ and let $0^\ell = 1$ for $\ell \leq 0$.
We have that 
$$A^n = \sum_\lambda  \sum_{j=0}^{\nu(\lambda) - 1} {n \choose j} \lambda^{n-j} (A -  \lambda I)^j E_\lambda.$$
\end{theorem}

Because $A$, and therefore $\nu(\lambda_i)$, is fixed we have thus established the limiting behavior of $A^n$.

\begin{corollary} \label{asymp growth for lin alg}
Let $A$ be a matrix with eigenvalues $\lambda_i$.
Let $\rho = \max_i\{\|\lambda_i\|\}$, and let $S = \{i : \|\lambda_i\| = \rho\}$.
Let $\nu = \max \{ \nu(\lambda_i) : i \in S \}$, and let $S \supseteq T = \{i : \|\lambda_i\| = \rho, \nu(\lambda_i) = \nu\}$.
For $\lambda_i \neq 0$, let $\widehat{E_i} = (A \lambda_i ^{-1} -  I)^{\nu - 1} E_{\lambda_i}$.
If $\rho \neq 0$, then we have that 
$$ \lim_{n \rightarrow \infty} \left( \frac{  A^n } { {n \choose \nu - 1} \rho^{n - \nu + 1} }  - \sum_{i \in T} \left(\frac{\lambda_i}{\rho}\right)^{n - \nu + 1} \widehat{E_i} \right)= 0. $$
\end{corollary}

Recall that Perron-Frobenius established that if the entries of $A$ are nonnegative real and $i \in S$, then dominant eigenvalue $\lambda_i$ must be $\rho$ times a root of unity.
So under these assumptions, $\sum_{i \in T} \left(\frac{\lambda_i}{\rho}\right)^{n - \nu + 1} \widehat{E_i}$ forms a periodic sequence.
If we divide both sides of Theorem \ref{limiting behavior} by $\rho^{n - \nu + 1}$ but not ${n \choose \nu - 1}$ and sum over $S$ instead of $T$, then we obtain once again the polynomials that Rothblum \cite{Rothblum_expansion_of_sums} uses to describe the growth of $A^n$ as in Section 3.

A separate application of Theorem \ref{limiting behavior} is to consider small values of $n$.
Specifically, we see that 
\begin{equation}\label{spread form of A}
A^1 = \sum_\lambda \lambda E_\lambda + (A - \lambda I) E_\lambda.
\end{equation}
This is a generalization of the spectral decomposition (also called the eigendecomposition).
A matrix is \emph{diagonalizable} if $\nu(\lambda_i) = 1$ for all $i$, and the spectral decomposition of a diagonalizable matrix $A$ is the canonical form $A = \sum_\lambda \lambda E_\lambda$.
We define $A_D = \sum_\lambda \lambda E_\lambda$ and $A_N = \sum_\lambda (A - \lambda I) E_\lambda$.
Because $(A - \lambda I)^{\nu(\lambda)}E_\lambda = 0$ and by Lemma \ref{define spectral projector}(1,2,4) , it follows that $A_N$ is nilpotent.
We have thus constructed $A = A_D + A_N$, such that $A_D$ is diagonalizable and $A_N$ is nilpotent.
Moreover, if $A$ is diagonalizable, then $A_N = 0$ and we observe the spectral decomposition as a special case.
Our partition of $A$ into two parts is consistent with the diagonalizable and nilpotent parts derived from the Schur decomposition of a matrix and the semi-simple and nilpotent components of the Jordon-Chevalley decomposition.

\subsection{Spectral Projectors as Eigenvectors} \label{spec proj as eig}
There is much unnecessary mystery around the $E_\lambda$.
Agaev and Chebotarev \cite{Agaev_Chebotarev} gave the following survey of results around $E_\lambda$, which testifies to how thoroughly it has been studied.

\begin{theorem} \label{survey}
Let $E_\lambda$ be the spectral projectors, as calculated in the proof to Theorem \ref{define spectral projector}.
Then $E_0 = Z$ if and only if any of the following hold: \\
(a) (Wei \cite{Wei}, Zhang \cite{Zhang}) $Z^2 = Z$, $A^{\nu(0)}Z = Z A^{\nu(0)} = 0$, and $rank(A^{\nu(0)}) + rank(Z) = n$,\\
(b) (Koliha and Stra\v{s}kraba \cite{Koliha_Straskraba}, Rothblum \cite{Rothblum_resolvent}) $Z^2 = Z$, $AZ = ZA$, and $A+ \alpha Z$ is nonsingular for all $\alpha \neq 0$,\\
(c) (Koliha and Stra\v{s}kraba \cite{Koliha_Straskraba}) $Z^2 = Z$, $AZ = ZA$, $A+ \alpha Z$ is nonsingular for an $\alpha \neq 0$, and $AZ$ is nilpotent,\\ 
(d) (Harte \cite{Harte}) $Z^2 = Z$, $AZ = ZA$, $AZ$ is nilpotent, and there exists matrices $U,V$ such that $AU = I - Z = VA$, \\
(e) (Hartwig \cite{Hartwig}, Rothblum \cite{Rothblum_drazin}) $Z = I - AA^D$, \\
(f) (Hartwig \cite{Hartwig}, Rothblum \cite{Rothblum_comp}) $Z = X(Y^*X)^{-1}Y^*$, where $X$ and $Y$ are the matrices whose columns make up a basis for the generalized $0$-eigenspace of $A$ and $A^*$, respectively, where $T^*$ is the conjugate transpose of $T$, \\
(g) (Agaev and Chebotarev \cite{Agaev_Chebotarev}) $Z = \prod_{\lambda \neq 0}\left(1 - (A/\lambda)^{\nu(0)}\right)^{\nu(\lambda)}$, and\\
(h) (folklore) $Z$ is the projection on $\ker(A^{\nu(0)})$ along $\mathcal{R}(A^{\nu(0)})$.
\end{theorem} 

We now turn to a major intuitive point of this manuscript: that the spectral projector $E_\lambda$ can be characterized as an outerproduct of generalized $\lambda$-eigenvectors.

\begin{theorem} \label{outerproduct}
Let $A$ be a matrix with eigenvalue $\lambda$ with algebraic multiplicity $m:=m(\lambda)$.
Let $v_{R,1}, \ldots, v_{R,m}$ and $v_{L,1}, \ldots, v_{L,m}$ denote arbitrary sets of left and right generalized $\lambda$-eigenvectors such that each set is linearly independent.
Let $V_R$ denote $n \times m$ matrix where column $i$ is $v_{R,i}$, and let $V_L$ denote a $m \times n$ matrix where row $i$ is $v_{L,i}$.
We have then that $V_LV_R$ is invertible and $E_\lambda = V_R (V_L V_R)^{-1} V_L$.
\end{theorem}
\begin{proof}
Let $e(x)$ be such that $E_\lambda = e(A)$, as in the proof to Theorem \ref{define spectral projector}.
Recall that $(A - \lambda I)^{\nu(\lambda)}E_\lambda = 0$.
This implies that the columns of $E_\lambda$ are generalized $\lambda$-eigenvectors of $A$; so $E_\lambda = V_R U_R$ for some $m \times n$ matrix $U_R$.
Because $E_\lambda^T = e(A)^T = e(A^T)$, we can apply a symmetric argument to say that $E_\lambda = U_L V_L$ for some $n \times m$ matrix $U_L$ (using the fact that $E_\lambda$ and $A$ commute).
So then we have that $E_\lambda = E_\lambda^2 = (V_R U_R)(U_L V_L) = V_R (U_R U_L) V_L$.
Repeating the same argument, we see that
$$ V_R (U_R U_L) V_L = E_\lambda = E_\lambda^2 = V_R (U_R U_L) V_L V_R (U_R U_L) V_L. $$

$U_R U_L$ and $V_L V_R$ are square matrices of size $m(\lambda)$, so each $E_\lambda$ has rank at most $m(\lambda)$.
Because $I = \sum_\lambda E_\lambda$, it must be that each $E_\lambda$ has rank exactly $m(\lambda)$.
Therefore $U_R U_L$ and $V_L V_R$ have full rank and are invertible.
We conclude that $U_R U_L = (V_L V_R)^{-1}$.
\end{proof}

Each space defined by generalized $\lambda$-eigenvectors is a linear subspace.  
Hence, if we consider $V_R' = V_R (V_L V_R)^{-1}$, then $V_R'$ is also a $n \times m$ matrix where column $i$ is $v_{R,i}'$, and the $v_{R,i}'$ form a basis of the generalized $\lambda$-eigenvectors.
We have $E_\lambda = V_R' V_L = \sum_i v_{R,i}' v_{L,i}$, and so $E_\lambda$ really is the outerproduct of $m(\lambda)$ (carefully chosen) generalized $\lambda$-eigenvectors.

With this deeper understanding of the $E_\lambda$, certain arguments become simpler.
First, we remind the reader about the Drazin inverse and Fredholm's Theorem.

The Drazin inverse of $A$, denoted $A^D$, is defined as the unique matrix such that $AA^D = A^DA$, $A^DAA^D = A^D$, and $A^{\nu(0) + 1}A^D = A^{\nu(0)}$.
There exist many characterizations of the Drazin inverse \cite{Stanimirovic_Djordjevic,Zhang,Ben-Israel_Greville,Rothblum_drazin,Wei,Wei_Wu,Wei_Qiao,Chen}.
We provide a spectral characterization of the Drazin inverse that highlights the relationship between the Drazin inverse, the inverse, and the spectral projectors.
This interpretation also appears as exercise 7.9.22 of \cite{Meyer}.

\begin{theorem} \label{drazin theorem}
If $A$ is invertible, then
$$ A^{-1} = \sum_{\lambda} \lambda ^{-1} \sum_{i=0}^{\nu(\lambda)-1} (I - A \lambda^{-1})^{i}E_\lambda . $$
The Drazin inverse of $A$ is 
$$ A^D = \sum_{\lambda \neq 0} \lambda ^{-1} \sum_{i=0}^{\nu(\lambda)-1} (I - A \lambda^{-1})^{i}E_\lambda . $$
\end{theorem}
\begin{proof}
Recall that $1 + (-x)^k = (1 + x)(1 - x + x^2 \cdots + (-x)^{k-1})$.
If $x$ is nilpotent and $x^k = 0$, then $1 + x$ and $\sum_{i=0}^{k-1}(-x)^i$ are inverses.
Because $(A - \lambda I)^{\nu(\lambda)}E_\lambda = 0$, let $x = A \lambda^{-1} - I$ when $\lambda \neq 0$ to see that
$$ \lambda(\lambda E_\lambda + (A - \lambda I) E_\lambda)^{-1} = \sum_{i=0}^{\nu(\lambda)-1} (-A \lambda^{-1} + I)^i E_\lambda .$$
Using (\ref{spread form of A}) and Theorem \ref{define spectral projector}, our stated expression for $A^D$ satisfies $AA^D = A^DA = \sum_{\lambda \neq 0} E_\lambda = I - E_0$.
The rest of the equations in the definition of Drazin inverse quickly follow.
\end{proof}

Fredholm's Theorem (see equation 5.11.5 of \cite{Meyer}) states that the orthogonal compliment of the range of $A$ is the null space of the conjugate transpose of $A$, and that the orthogonal compliment of the null space of $A$ is the range of the conjugate transpose of $A$.
The statement below quickly follows from applying Fredholm's theorem to $(A - \lambda I)^{\nu(\lambda)}$; we give an independent proof.

\begin{proposition} [Fredholm's Theorem (special case)] \label{orthogonal}
Let $A$ be a matrix with 
a generalized right $\lambda$-eigenvector $v_R$ and a generalized left $\lambda'$-eigenvector $v_L'$ .
If $\lambda \neq \lambda'$, then $v_L' v_R = 0$.
\end{proposition}
\begin{proof}
Let $v_R$ have index $k$ and let $v_L'$ have index $k'$.
Consider the term $v_L' (A - \lambda' I)^{k'} v_R$.
By definition of $k'$ this term equals $0$.
We claim that this term equals $c v_L' v_R$ for some $c \neq 0$, which will prove the proposition.
We proceed by induction on $k$.
If $k = 1$, then $v_R$ is en eigenvector and $(A - \lambda' I)^{k'} v_R = (\lambda - \lambda')^{k'} v_R$.
Therefore the claim follows with $c = (\lambda - \lambda')^{k'} \neq 0$ when $\lambda \neq \lambda'$.

Now we proceed with induction; assume that $v_L' v' = 0$ for all generalized right $\lambda$-eigenvectors $v'$ with index $j$ when $j < k$.
We have that 
\begin{eqnarray*}
v_L' (A - \lambda' I)^{k'} v_R 		& = & 	v_L' (A - \lambda I + (\lambda - \lambda')I)^{k'} v_R \\
					& = &	v_L' \sum_{i = 0}^{k'} (\lambda - \lambda')^{k' - i} (A - \lambda I)^{i} v_R \\
					& = &  \sum_{i = 0}^{k'} (\lambda - \lambda')^{k' - i} v_L'(A - \lambda I)^{i} v_R  .
\end{eqnarray*}
Recall that $(A - \lambda I)^{i} v_R$ is a generalized right $\lambda$-eigenvector with index $k - i$ (in this case, a nonpositive index refers to the vector $0$).
So by induction, $v_L'(A - \lambda I)^{i} v_R = 0$ when $i > 0$.
Thus we have that $v_L' (A - \lambda' I)^{k'} v_R = (\lambda - \lambda')^{k'} v_L' v_R$, and the proposition follows.
\end{proof}

Next, we return to the characterizations of $E_0$.

\textit{Proof of Theorem \ref{survey}.} Clearly the $E_0$ as we have defined satisfy conditions (a), (b), (c), and (d).
If $AZ = ZA$, $Z^2 = Z$, and $AZ$ is nilpotent, then $A^nZ = 0$, so the columns of $Z$ must be generalized $0$-eigenvectors of $A$.
By considering $((AZ)^T)^n = ((AZ)^n)^T$, a symmetrical statement can be said about the rows of $Z$ and the generalized left $0$-eigenvectors.
Following the argument of Theorem \ref{outerproduct}, we see that $Z$ must be $E_0$ if $Z^2 = Z$, $AZ = ZA$, $AZ$ is nilpotent, and there is some condition that implies $rank(Z) \geq m(0)$.
Koliha and Stra\v{s}kraba \cite{Koliha_Straskraba} gave a short proof (maybe 6 lines after all the references are combined) that the conditions of (b) imply that $AZ$ is nilpotent.
Therefore the equivalence of (a), (b), (c), and (d) follow.

Theorem \ref{drazin theorem} and Theorem \ref{define spectral projector}(3) imply (e).
The equivalence of (f) is trivial.
In Theorem \ref{define spectral projector} we may assume that $e_i(x) = \prod_{j \neq i}\left(1 - \left(\frac{x - \lambda_i}{\lambda_j - \lambda_i}\right)^{\nu(\lambda_i)}\right)^{\nu(\lambda_j)}$, and so (g) is equivalent.
Part (h) follows from Fredholm's Theorem.
\hfill $\square$

Our final result of this subsection is the one that surprised us the most.
It is natural that if eigenvectors are the correct answer in the special case of diagonalizable matrices, then generalized eigenvectors may be the correct solution for general matrices.
However, while we have needed generalized eigenvectors in our arguments, our next result is that eigenvectors are sufficient for the general matrix!

\begin{theorem}\label{general to regular}
Using the notation of Corollary \ref{asymp growth for lin alg}, if $i \in T$, then $\widehat{E_{\lambda_i}}$ is an outerproduct of eigenvectors.
\end{theorem}
\begin{proof}
Recall that $\widehat{E_{\lambda_i}} = (A - \lambda_i I)^{\nu(\lambda_i)-1}E_{\lambda_i}$.
By Lemma \ref{define spectral projector}(4), we also have that $\widehat{E_{\lambda_i}} = E_{\lambda_i} (A - \lambda_i I)^{\nu(\lambda_i)-1}$.
By Theorem \ref{outerproduct} and the discussion afterwards, $E_{\lambda_i}$ is an outerproduct of generalized $\lambda_i$-eigenvectors.
That is, $E_{\lambda_i} = \sum_{j=1}^{m(\lambda_i)} v_{R,j} v_{L,j} $, where $v_{R,1}, \ldots, v_{R,m(\lambda_i)}$ are right generalized $\lambda_i$-eigenvectors and $v_{L,1}, \ldots, v_{L,m(\lambda_i)}$ are left generalized $\lambda_i$-eigenvectors.
Therefore 
$$\widehat{E_{\lambda_i}} =  \sum_{j=1}^{m(\lambda_i)} \left((A - \lambda_i I)^{\nu(\lambda_i)-1} v_{R,j}\right) v_{L,j} =  \sum_{j=1}^{m(\lambda_i)} v_{R,j} \left(v_{L,j} (A - \lambda_i I)^{\nu(\lambda_i)-1}\right).$$

Let $w_{R,j} = (A - \lambda_i I)^{\nu(\lambda_i)-1} v_{R,j}$ and $w_{L,j} = v_{L,j} (A - \lambda_i I)^{\nu(\lambda_i)-1}$.
If $v_{R,j}$ has index $k$, then $w_{R,j}$ has index $k - \nu(\lambda_i) + 1$ (where a nonpositive index indicates the zero vector).
By the definition of $\nu$, we have that $k \leq \nu(\lambda_i)$.
Therefore each $w_{R,j}$ has index at most one, and so $w_{R,j}$ is either a $\lambda_i$-eigenvector or the zero vector.
Symmetrically, $w_{L,j}$ is also either a $\lambda_i$-eigenvector or the zero vector.
Thus, the columns (rows) of $\widehat{E_{\lambda_i}}$ are contained in the span of the right (left) $\lambda_i$-eigenvectors of $A$.

That $\widehat{E_{\lambda_i}} = \sum_{j=1}^{t} u_{R,j} u_{L,j} $, where $t \leq m(\lambda_i)$ and each $u_{R,j}$ ($u_{L,j}$) is a right (left) $\lambda_i$-eigenvectors of $A$, now follows from Claim \ref{everything is an outerproduct}.
\end{proof}

\subsection{Spectral Projector as the Inverse of a Singular Matrix} \label{adjugate section}

In this section we describe the connection between spectral projectors and the adjugate matrix.
The adjugate of a matrix is the transpose of the cofactor matrix.
The row $i$ column $j$ entry of the cofactor matrix of $M$, denoted $\cof(M)$, is $(-1)^{i+j}$ times the determinant of the minor of $M$ after row $i$ and column $j$ are removed.
There are may properties of $\adj(A)$ that are well-known.
For example, $A\adj(A) = |A| I$, $\adj(A)$ is a polynomial in $A$ and the trace of $A$, and $\adj(A) = 0$ if the rank of $A$ is at most $2$ less than the dimension of $A$.

Our results will follow easier once we establish that the cofactor function is a homomorphism with the multiplication operation.
That is, $\cof(M_1 M_2) = \cof(M_1) \cof(M_2)$.
Many basic tutorials of linear algebra on the Internet state that $\adj(M_1 M_2) = \adj(M_2) \adj(M_1)$, but we have yet to find a proof that does not begin by assuming that both $M_1$ and $M_2$ are invertible.
By the above properties for the adjugate, the relation clearly follows unless at least one of $M_1$ or $M_2$ has rank exactly $1$ less than the dimension.
But this criteria is satisfied by large classes of matrices, such as the Laplacian of any connected graph.

To prove that $\cof(M_1 M_2) = \cof(M_1) \cof(M_2)$, we will define the family of \emph{extended elementary matrices}.
The elementary matrices represent the different operations used while transforming a matrix into row echelon form: row addition, row multiplication, and row switching.
Any invertible matrix can be written as a product of elementary matrices.
The family of extended elementary matrices is the family of elementary matrices plus the ability to perform row multiplication with a scaling factor of $0$.

\begin{lemma} \label{Yanceys lemma}
Any matrix can be written as a product of extended elementary matrices.
\end{lemma}
\begin{proof}
Let $M$ be a matrix that we wish to represent as a product of extended elementary matrices.
Suppose the dimension of $M$ is $n$, and let $M_1, M_2, \ldots, M_n$ be the rows of $M$.
Let $I$ be a maximum set of rows that are linearly independent.
So for each $M_j \notin I$ there exists a set of coefficients $c_i$ such that $M_j = \sum_{M_i \in I} c_i M_i$.
Let $B$ be a linearly independent basis for $\mathbb{C}^n$ that includes $I$, and let $M'$ be a matrix whose rows are $B$.

$M'$ can be written as a product of elementary matrices $R_1 R_2 \cdots R_t$.
We transform $M'$ into $M$ using the following operations:\\
(1) perform row multiplication on each row in $B \setminus I$ with a scaling factor of $0$, and then
(2) construct each row in $M_j \in \{M_1 \ldots, M_t\} \setminus I$ using row addition based on the equations $M_j = \sum_{M_i \in I} c_i M_i$.\\
Each of the above operations can be represented by a member of the extended elementary matrices.
Thus $R_1 R_2 \cdots R_t$ can be grown into 
$$R_k' \cdots R_2' R_1' R_1 R_2 \cdots R_t = M.$$
\end{proof}

We imagine that Lemma \ref{Yanceys lemma} will make many other proofs easier.
For example, one can use it to quickly prove that the determinant is a multiplicative homomorphism.
It certainly is a crucial simplification for proving Proposition \ref{cof homo}.

\begin{proposition} \label{cof homo}
For any matrices $M_1$ and $M_2$ we have that $\cof(M_1 M_2) = \cof(M_1) \cof(M_2)$.
\end{proposition}
\begin{proof}
We will prove that for any extended elementary matrix $R$, we have that $\cof(R M) = \cof(R) \cof(M)$.
By Lemma \ref{Yanceys lemma}, repeated application of this statement will prove the proposition.
Moreover, row switching can be represented as a product of row addition and row multiplication, so we further assume that $R$ either represents row addition or row multiplication (by possibly a scaling factor of $0$).
Another trivial reduction is that we may assume the coefficient of row addition is $1$.

\textbf{Case 1:} row addition.  Let $R = R_{i,j}$ be the identity matrix except for the entry in row $i$, column $j$ ($i \neq j$), which equals $1$.
The matrix $R M$ is the matrix $M$, except that for each $1 \leq t \leq n$, the row $i$ column $t$ entry of $R M$ is the sum of the row $i$ column $t$ and the row $j$ column $t$ entry of $M$.
Also, $\cof(R)$ is the identity matrix except for the entry in row $j$, column $i$, which equals $-1$.
Therefore the matrix $\cof(R) \cof(M)$ is the matrix $\cof(M)$, except that for each $1 \leq t \leq n$, the row $j$ column $t$ entry of $\cof(R) \cof(M)$ is the entry in row $j$ column $t$ minus the entry in the row $i$ column $t$ entry of $\cof(M)$.

Let $C_{k,\ell}$ be the minor of $\cof(R_{i,j} M)$ used to determine the value in row $k$ column $\ell$ of $\cof(R_{i,j} M)$.
If $k = i$, then this minor is the same used to calculate the value in row $k$ column $\ell$ of $\cof(M)$.
Now suppose $k \neq i$.

Recall that $|C| = |A| + |B|$ if there exists an $r$ such that the entries of $A,B,C$ are equal in all rows except $r$, and row $r$ of $A$ and row $r$ of $B$ sum to row $r$ of $C$.
We will apply this statement with $r = i$ to compare  $\cof(R_{i,j} M)$ against $\cof(M)$.
Because $k \neq i$, we see that $|C_{k,\ell}|$ can be calculated as the sum of two matrix determinants.
The first matrix, called $A_{k,\ell}$, is the same minor used to calculate the value in row $k$ column $\ell$ of $\cof(M)$.
The second matrix, called $B_{k,\ell}$, is almost the same minor used to calculate the value in row $k$ column $\ell$ of $\cof(M)$, except with row $i$ replaced with contents of row $j$.
If $k \neq j$, then $B_{k,\ell}$ contains the contents of row $j$ of $M$ twice (once in row $i$ and once in row $j$), and therefore $|B_{k,l}| = 0$.
If $k = j$, then $B_{k,\ell}$ is $C_{i,\ell}$, except that row $j$ has been permuted into the location of row $i$, with the rows in between them shifted up/down accordingly.
This permutation will multiply the determinant by $(-1)^{i-j}$.

The previous paragraph only calculates the determinants of respective matrix minors here; we must also account for the $(-1)^{i+j}$ term in the cofactor matrix.
In particular, we will see some cancellation: $(-1)^{i-j}(-1)^{j-i} = 1$.
This concludes the proof to the proposition for Case 1.

\textbf{Case 2:} row multiplication. 
This case follows easily.
\end{proof}

Now that we have established that cofactors respect multiplication, we are prepared to describe the adjugate of a matrix with rank $1$ less than the dimension.
Note that this criteria is equivalent to the assumption $m(0) = \nu(0)$.

\begin{theorem} \label{spectral as adj}
Let $A$ be a matrix with $m(0) = \nu(0)$.
Let $v_L, v_R$ be the unique left, right $0$-eigenvectors of $A$, normalized to equal the appropriate vectors in the conjugation matrix of the Jordan Normal Form. 
We have that 
$$ \adj(A) = v_R v_L (-1)^{m(0)-1} \prod_{\lambda \neq 0} \lambda^{m(\lambda)} = (-A)^{\nu(0)-1}E_0 \prod_{\lambda \neq 0} \lambda^{m(\lambda)}. $$
\end{theorem}
\begin{proof}
The first part of our proof does not assume that the geometric multiplicity of $0$ is $1$.
The claims are easier to validate in this manner.

Write $A$ in Jordon Normal Form, so that $A = Q^{-1}JQ$, where $J$ is block diagonal with each block being a Jordon block, each row of $Q$ is a left generalized $\lambda$-eigenvector of $A$, and each column of $Q^{-1}$ is a right generalized $\lambda$-eigenvector of $A$ (where $\lambda$ is the eigenvalue of the associated Jordon block).
By Proposition \ref{cof homo}, we know that $\adj(A) = \adj(Q)\adj(J)\adj(Q^{-1})$.
Because $Q$ is invertible, we know that $\adj(Q) = |Q|Q^{-1}$ and $\adj(Q^{-1}) = |Q|^{-1}Q$.
Therefore $\adj(A) = Q^{-1} \adj(J) Q$.

Suppose $J$ is composed of Jordon blocks $J_1, J_2, \ldots J_k$.
Let $\lambda_i$ denote the eigenvalue in $J_i$, and let $n(J_i)$ denote the dimension of $J_i$.
It is easy to see that if row $i$ column $j$ is not in any of the $J_\ell$, then the row $i$ column $j$ entry of $\adj(J)$ is $0$.
It follows that $\adj(J)$ is block diagonal, with blocks $T_1, T_2, \ldots, T_k$, where $T_i = \adj(J_i) \prod_{j \neq i} \lambda_j^{n(J_j)}$.
If $A$ is invertible, then this construction of $\adj(J)$ is consistent with $\adj(A) = |A|A^{-1}$.
The existence of two Jordon blocks whose eigenvalue is $0$ is equivalent to geometric multiplicity of $0$ being at least $2$, which is equivalent to the rank of $A$ being at most $2$ less than the dimension of $A$.
Therefore our construction is consistent with the fact that $\adj(A) = 0$ in this case.
Now we will use the assumption of the theorem: that the geometric multiplicity of $0$ is $1$.

Without loss of generality, assume $\lambda_1 = 0$ and $\lambda_\ell \neq 0$ for $\ell > 1$.
As we have noted, $T_\ell = 0$ for all $\ell > 1$.
It is clear that $\adj(J_1)$ is $0$ in all entries except the entry in row $1$ column $n(J_1)$, which is $(-1)^{n(J_1)+1} = (-1)^{\nu(0)-1}$.
Because $J_1$ is the only Jordon block corresponding to eigenvalue $0$, we have that we have that $n(J_1) = m(0)$.
Therefore $\adj(J)$ is zero in all entries, except the entry in row $1$ column $n(J_1)$, which is $(-1)^{\nu(0)-1}\prod_{j \neq 1} \lambda_j^{n(J_j)}$.
Row $1$ corresponds to column $1$ of $Q^{-1}$, which is the right $0$-eigenvector of $A$, which is $v_R$.
Column $n(J_1)$ corresponds to row $n(J_1)$ of $Q$, which is the left $0$-eigenvector of $A$, which is $v_L$.
In particular,  $Q^{-1} \adj(J) Q$ produces the outerproduct $v_R v_L$ times the coefficient $(-1)^{\nu(0)-1} \prod_{\lambda \neq 0} \lambda^{m(\lambda)}$.

The final step of the proof is to show that $v_R v_L = A^{\nu(0)-1}E_0$.
Let $I_0$ be the matrix that is $1$ on the diagonal wherever $J$ is $0$ on the diagonal, and $I_0$ is $0$ everywhere else.
Because of the characterization of $Q$ and $Q^{-1}$ as generalized eigenvectors and Theorem \ref{outerproduct}, we see that $Q^{-1}I_0Q = E_0$.
As $A^{\nu(0)-1} = Q^{-1}J^{\nu(0)-1}Q$, it is an easy calculation to see that $(-1)^{\nu(0)-1}\adj(J_1)$ is the restriction of $J^{\nu(0)-1}I_0$ to the rows and columns that contain the Jordon block $J_1$.
\end{proof}

As we have mentioned, if $A$ is invertible then $\adj(A) = A^{-1}|A| = A^{-1} \prod_\lambda \lambda^{m(\lambda)}$.
This is still intuitively true if $m(0) > \nu(0)$, as $\adj(A) = |A| = 0$ in this case.
Theorem \ref{spectral as adj} is the last case necessary to establish that the intuition behind $\adj(A) = A^{-1}|A|$ is true in general.

\begin{corollary} \label{adjugate theorem}
If $A$ is invertible, then
$$ A^{-1} = \sum_{\lambda}  \sum_{i=0}^{\nu(\lambda)-1} \lambda ^{-1 - i} (\lambda I - A)^{i}E_\lambda . $$
The adjugate of $A$ is 
$$ \adj(A) = \sum_{\lambda} \sum_{i=0}^{\nu(\lambda)-1} \left( \prod_{\lambda_* \neq \lambda}\lambda_*^{m(\lambda_*)} \right) \lambda^{m(\lambda)-1-i} (\lambda I - A)^{i}E_\lambda. $$
\end{corollary}


The proof of Theorem \ref{spectral as adj} can be adapted to calculate $(A-\lambda I)^{\nu(\lambda)-1}E_\lambda$ without requiring the assumption that $\nu(0) = m(0)$.
We will also work directly with the matrix inverse instead of the adjugate.
Recall that if $\lambda_i \neq 0$, then $\lambda_i^{1-\nu(\lambda_i)}(A - \lambda_i I)^{\nu(\lambda_i)-1}E_{\lambda_i} = \widehat{E_i}$.
The following statement is contained in Corollary 3.1 of \cite{Meyer_paper}; the proof is new (to our knowledge).

\begin{theorem}
Let $A$ be a matrix with eigenvalue $\lambda$.
Let $x_1, x_2, \ldots $ be a sequence that converges to $\lambda$ and such that $A- x_iI$ is invertible for all $i$.
Under these assumptions, 
$$ \lim_{i \rightarrow \infty} (\lambda - x_i)^{\nu(\lambda_i)}(A - x_i I)^{-1} = (A - \lambda I)^{\nu(\lambda)-1}E_{\lambda}. $$
\end{theorem}
\begin{proof}
We start with the same context as the proof to Theorem \ref{spectral as adj}.
Let $A$ be written in Jordon Normal Form as $Q^{-1}JQ$, where $J$ has Jordon blocks $J_1, \ldots, J_k$.
Let $J^{(i)} = J - x_i I$, so that $Q^{-1}J^{(i)}Q$ is Jordon Normal Form for $A - x_i I$.
Let $J_1^{(i)}, \ldots, J_k^{(i)}$ be the Jordon blocks for $J^{(i)}$.
$J_j$ and $J_j^{(i)}$ are essentially the same, except that if $\lambda'$ is the eigenvalue of Jordon block $J_j$, then $(\lambda' - x_i)$ is the eigenvalue for $J_j^{(i)}$.

If $A$ is invertible, then $A^{-1} = Q^{-1}J^{-1}Q$.
If $J^{-1}$ exists, then it is block diagonal, where the blocks are $(J_1)^{-1}, \ldots, (J_k)^{-1}$.
If the eigenvalue of $J_j$ is $\lambda' \neq 0$, then the entry in row $s$ column $t$ of $(J_j)^{-1}$ is $0$ if $s > t$ and $-(-\lambda)^{s-t-1}$ if $s \leq t$.
\end{proof}

\section{Symbolic Dynamics} \label{dynamics section}
As a matter of notation, recall that the set of coordinates in an $n$-dimensional vector are in bijection with the set of vertices of $D$.
Hence, each vector can be thought of as a function $v:V \rightarrow \mathbb{C}$, and the adjacency matrix is thought of as an operator on such functions.
This context will allow us to simplify our arguments by using phrases like \emph{the support of vector $v$}, which is the set of coordinates $i$ such that $v(i) \neq 0$.
If we consider some sub-digraph $D'$ of digraph $D$, then we may transform vectors (matrices) over $D$ into vectors (matrices) over $D'$ by \emph{restricting the domain} or \emph{inducing $D'$ on $D$} as another phrase for a matrix minor.
When we have stated a partition of $D$ as $D_1, \ldots, D_k$, we use the notation $x^{(i)}$ to denote vector (matrix) $x$ induced on $D_i$.

Let $D'$ be the matrix $D$ restricted to domain $S$, and let $D''$ be the $S$-mask of $D$.
The spectral properties of $D'$ and $D''$ are almost identical.
Generalized eigenvectors of $D'$ can be transformed into generalized eigenvectors in $D''$ by placing $0$ in the coordinates not in $S$.
The other generalized eigenvectors of $D''$ are $0$-eigenvectors whose support is the compliment of $S$.
We will make strong attempts to keep the distinction between an induced matrix and a mask of a matrix, but occasionally we will swap between the two.

If the only eigenvalue of $A$ is $0$, then $A$ is nilpotent.
So for the rest of this paper, assume that the spectral radius of $A$ is positive.

Let $D$ be a digraph with irreducible components $D_1, \ldots D_t$.
We assume that the vertices are ordered so that the adjacency matrix of $D$ is in \emph{Frobenius normal form}, which we will now explain.
If vertex sets $S_1$ and $S_2$ each induce an irreducible subdigraph of $D$, and the sets of $(S_1,S_2)$-walks and $(S_2,S_1)$-walks are each non-empty, then $S_1 \cup S_2$ is a subset of an irreducible component.
Therefore for the rest of the paper we assume that the irreducible components are ordered such that the set of $(V(D_j),V(D_i))$-walks is empty when $j > i$, unless stated otherwise.
Moreover, we assume that if vertices $x,y$ satisfy $x \in D_i$ and $y \in D_j$, then $x < y$ implies that $i \leq j$.
Under these assumptions, the adjacency matrix $A$ of $D$ is then in Frobenius normal form (also known as \emph{block upper triangular form}), where the blocks correspond to the irreducible components.

The results in Section \ref{structural ergodic theory} are explicitly for general matrices, and most results in Section \ref{dynamics section} can be generalized to this setting.

An \emph{edge shift} of a digraph $D$ is the family of biinfinite walks in $D$.
The connection between edge shifts and the digraph representing a regular language is well-known (for example, see \cite{Lind_Marcus}).
In the following  we will study how modifications to $D$ will affect the associated shift.
Our main goal will be to construct a family of graphs whose associated edge shifts roughly approximates a partition of the edge shift of $D$.
Moreover, the set of walks in each member of the family should be easy to study.

\subsection{Background}

A digraph is $p$-cyclic if there exists a partition of the vertex set into disjoint classes $P_0, P_1, \ldots, P_{p-1}$ such that each arc $(u,v)$ such that $u \in P_i$ also satisfies $v \in P_{i+1}$ where the indices are taken modulo $p$.
The period of a digraph $D$ is the maximum $p$ such that $D$ is $p$-cyclic.
A digraph is \emph{aperiodic} if its period is $1$.
In particular, if arc $(v,v) \in E$ for any vertex $v$, then the digraph is aperiodic.

We now recall several facts from Perron-Frobenius theory.
Let $\rho$ be the spectral radius of adjacency matrix $A$ of irreducible digraph $D$ with period $p$.
The eigenvalues of $A$ include $w \rho$, where $w$ ranges over the solutions of $x^p - 1 = 0$, and these are exactly the dominant eigenvalues of $A$.
For each solution $w$ of $x^p - 1 = 0$ the eigenvalue $w \rho$ has algebraic multiplicity $1$.
Let $u_w$ be a right $(w\rho)$-eigenvector of $A$.
The eigenvector $u_1$ has positive real entries for all coordinates; if coordinate $i$ of $u_w$ corresponds to a vertex in $P_j$, then coordinate $i$ of $u_w$ equals $w^j$ times coordinate $i$ of $u_1$.
If $v_w$ is a left $(w\rho)$-eigenvector of $A$, then coordinate $i$ of $v_w$ equals $w^{-j}$ times coordinate $i$ of $v_1$.

The following is a standard result; it also clearly follows from Theorem \ref{limiting behavior} and Theorem \ref{outerproduct}.
We include the proof here because we build off of it in future results.

\begin{proposition}[Theorem 4.5.12 in \cite{Lind_Marcus}] \label{primitive behaves nicely}
Let $A$ be an adjacency matrix for a primitive digraph with spectral radius $\rho$.
Let $v_L, v_R$ be left, right $\rho$-eigenvectors of $A$, normalized such that $v_L v_R = 1$.
For each row vector $w_L$ and column vector $w_R$ there exists $C, \epsilon > 0$ such that 
$$ w_L A^m w_R = \left( (w_L v_R)(v_L w_R) + p(m) \right) \rho^m, $$
where $p(m) < C ( 1 - \epsilon)^m$.
\end{proposition}
\begin{proof}
Let $n$ be the dimension of $A$, and let $U$ be the $n-1$ dimensional subspace of row vectors orthogonal to $v_R$ (so $U = \{w: w v_R = 0\}$).
Because $v_R$ is an eigenvector and $\rho 0 = 0$, we have that $U$ is closed from multiplication on the right by $A$ (so $U A \subseteq U$).
Let $u_1, u_2, \ldots, u_{n-1}$ be a basis for $U$.
Also note that $v_L \notin U$, as $v_L, v_R$ are all positive reals (and so our normalization $v_L v_R = 1$ is always possible).
So $v_L, u_1, \ldots, u_m$ is a basis for $\mathbb{R}^n$.
Let $w_L = a_v v_L + \sum_{i=1}^{n-1} a_i u_i$.
To calculate $a_v$, notice that $w_L - a_v v_L \in U$, so $(w_L - a_v v_L) v_R = 0$.
By assumption $v_L v_R = 1$, so $a_v = w_L v_R$.

Let $\epsilon > 0$ be such that $\lambda_* = \rho(1 - \epsilon)$ is larger than the second largest eigenvalue of $A$.
The eigenvalues of $\lambda_*^{-1} A$ restricted to $U$ are all strictly less than $1$ and so the restriction of $(\lambda_* ^{-1} A)^m$ to $U$ converges to the zero matrix as $m$ increases.
As a consequence, if we let $p_*(m) = \left(\sum_{i=1}^{n-1} a_i u_i \right) (\lambda_*^{-1}A)^m w_R$, then $p_*(m) \rightarrow 0$.
Observe,
\begin{eqnarray*}
  w_L A^n w_R	& = &	\left(a_v v_L + \sum_{i=1}^{n-1} a_i u_i \right) A^n w_R \\
		& = &	(w_L v_R) v_L A^n w_R + \lambda_*^n \left(\sum_{i=1}^{n-1} a_i u_i \right) (\lambda_*^{-1}A)^n w_R \\
		& = & 	\lambda^n \left( (w_L v_R) (v_L w_R) + (1-\epsilon)^n p_*(n) \right).
\end{eqnarray*}
\end{proof}

\begin{definition}
Let $D$ be a digraph.
The digraph $D^r$ has the same vertex set as $D$, and the arcs from vertex $a$ to vertex $b$ are in bijection with the set of $(a,b)$-walks in $D$ of length $r$.
The digraph $D^r$ is called \emph{the $r$ power of $D$}.
\end{definition}

If $A$ is the adjacency matrix of $D$, then $A^r$ is the adjacency matrix of $D^r$.
Recall that $\lambda^r$ is an eigenvalue of $A^r$ if $\lambda$ is an eigenvalue of $A$.
Moreover, the set of $\lambda$-eigenvectors of $A$ are $\lambda^r$-eigenvectors of $A^r$.
We are interested in powers of a matrix, because if $D$ is irreducible and has period $p$, then there exists a unique dominant eigenvalue of $D^p$ (unfortunately, with geometric and algebraic multiplicity $p$).

If $D$ is irreducible and has periodic classes $P_0, \ldots, P_{p-1}$, then $D^p$ has $p$ connected components corresponding to the periodic classes.
Let $D_i$ be the subdigraph of $D^p$ induced on vertex set $P_i$.
It is known (see Section 4.5 of \cite{Lind_Marcus}) that $D_i$ is primitive for each $i$.
Furthermore, $A^p$ is a matrix that is block diagonal: the $(i,j)$ entry is nonzero only if vertices $i$ and $j$ are in the same periodic class.
Then $(A^p)^{(i)}$ is the adjacency matrix for a $D_i$.
By the above, $\left((A^p)^{(i)}\right)^n$ denotes all walks of length $pn$ in $A$ that start (or end) at $P_i$.

The following result is Exercise 4.5.14 in \cite{Lind_Marcus} after the correction recorded in the textbook's errata.
As the techniques are repeated in later arguments, we again include the proof here.
Specifically, our proof to Proposition \ref{periodic behaves nicely} is a light introduction into our plan to break an edge shift into digestible chunks.

\begin{proposition} \label{periodic behaves nicely}
Let $D$ be an irreducible digraph with period $p$ and adjacency matrix $A$ with dominant eigenvalue $\lambda$.
Let $v_L, v_R$ be left, right $\lambda$-eigenvectors of $A$, normalized such that $v_L v_R = p$.
For a vector $v$, let $v^{(i)}$ denote the subvector induced on periodic class $i$.
For each row vector $w_L$, column vector $w_R$, and integer $k$ there exists $C, \epsilon > 0$ such that 
$$ w_L A^{pm + k} w_R = \left(\sum_{i=1}^p (w_L^{(i)} v_R^{(i)})(v_L^{(i+k)}w_R^{(i+k)}) + q(m) \right) \rho^{pm+k}, $$
where $q(m) < C ( 1 - \epsilon)^m$.
\end{proposition}
\begin{proof}
For the extent of this proof, let $w^{(i)}$ denote the $P_i$-mask of $w$ rather than the induced sub-digraph/matrix minor.
Note that the statement of the proposition is equivalent.

So $w^{(i)}$ is an $n$-dimensional vector whose support is contained by the vertex set $P_i$ (and the values of $w^{(i)}$ in the coordinates in $P_i$ match the values in $w$). 
When we use this notation, the index $i$ is taken modulo $p$.
By the cyclic nature of $A$, for all row vectors $w$, we have that $w^{(i)} A^k = (w A^k)^{(i+k)}$.
Symmetrically, for all column vectors $w$, we have that $(A^k w)^{(i)} = A^k w^{(i+k)}$.
Furthermore, $\sum_i v^{(i)} w^{(i)} = v w$ for all row vectors $v$ and column vectors $w$, as $v^{(i)}w^{(j)} = 0$ when $i \neq j$ due to disjoint support.

First, we claim that $v_L^{(i)} v_R^{(i)} = 1$ for all $i$.
By the definition of eigenvector, we have that $A v_R^{(i+1)} = (A v_R)^{(i)} = \rho v_R^{(i)}$ and symmetrically $v_L^{(i)} A = \rho v_L^{(i+1)}$.
Therefore $\rho v_L^{(i)} v_R^{(i)} = v_L^{(i)} A v_R^{(i+1)} = \rho v_L^{(i+1)} v_R^{(i+1)}$.
Because $\rho \neq 0$, we have that  $v_L^{(i)} v_R^{(i)} = v_L^{(j)} v_R^{(j)}$ for all $i,j$.
The claim then follows from $v_L v_R = p$.

Consider the term $w_L^{(i)}A^{pm} w_*^{(j)}$, where $w_* = A^{k}w_R$.
These terms allow us to split our final goal into smaller parts: $w_L A^{pm + k} w_R = \sum_i \sum_j w_L^{(i)} A^{pm} w_* ^{(j)}$.
Because the indices are taken modulo $p$, we have that $w_L^{(i)}A^{pm} w_*^{(j)} = (w_L A^{pm})^{(i)} w_*^{(j)}$, which equals zero when $i \neq j$.
So we can restrict our attention to the case that $i = j$.
Recall that $A^p$ forms $p$ primitive digraphs whose vertex sets are the periodic classes $P_0,\ldots,P_{p-1}$.
Because these subdigraphs are disconnected components, we have that $w_L^{(i)} A^{pm} w_* ^{(i)} = w_L^{(i)} \left((A^p)^{(i)}\right)^{m} w_* ^{(i)} $.
Now apply Proposition \ref{primitive behaves nicely} to $\left((A^p)^{(i)}\right)^{m}$ to see that 
\begin{eqnarray*}
w_L^{(i)} A^{pm} w_* ^{(i)} 	& = & \left( (v_L^{(i)} w_*^{(i)}) (w_L^{(i)} v_R^{(i)}) + q(n) \right) (\rho^p)^m \\
				& = & \left( (v_L^{(i)} (A^{k}w_R)^{(i)}) (w_L^{(i)} v_R^{(i)}) + q(n) \right) \rho^{pm} \\
				& = & \left( (v_L^{(i)} A^k w_R^{(i+k)}) (w_L^{(i)} v_R^{(i)}) + q(n) \right) \rho^{pm} \\
				& = & \left( (v_L^{(i+k)} w_R^{(i+k)}) (w_L^{(i)} v_R^{(i)}) + q(n)\rho^{-k} \right) \rho^{pm + k}.
\end{eqnarray*}
\end{proof}

\subsection{ Structural Results} \label{structural ergodic theory}

Recall that a generalized right $\lambda$-eigenvector with index $\nu$ is a vector $w$ such that $(A-\lambda I)^{\nu - 1}w \neq 0$ and $(A-\lambda I)^{\nu}w = 0$.
In the following, we consider the $0$ vector to be the unique vector with index $0$.

\begin{lemma} \label{original to components}
Let $M$ be a general matrix with associated digraph $D$ that has irreducible components $D_1, \ldots D_t$.
Let $M_1, \ldots, M_t$ be the submatrices of $M$ corresponding to $D_1, \ldots D_t$.
For a vector $v$, let $v^{(j)}$ be the sub-vector induced on $M_j$.
Let $v$ be a fixed generalized right $\lambda$-eigenvector of $A$, and let $i$ be the largest index such that $v^{(i)} \neq 0$ (we are assuming that $D$ is in Frobenius normal form).
Under these conditions, if $v$ has index $\nu$, then $v^{(i)}$ is a generalized right $\lambda$-eigenvector of $M_i$ with index at most $\nu$.
\end{lemma}
\begin{proof}
Under the assumptions of the lemma, $M$ is in \emph{block upper triangular form}.
In other words, $M$ can be represented by a $t\times t$ matrix $B$ whose row $j$ column $j'$ entry $b_{j,j'}$ satisfies the following properties: $b_{j,j'}$ is a rectangular matrix, $M_{j,j} = M_j$, and $b_{j,j'}$ is a matrix of all zeroes when $j > j'$.

Let $v$ be a generalized $\lambda$-eigenvector for $M$ with index $\ell$.
Consider the rectangular matrix $R_i = (b_{i,1}, b_{i,2}, b_{i,3}, \ldots, b_{i,t})$, which is the set of rows of $M$ that correspond to $D_i$.
The product $R_i v$ is $\sum_{j=1}^t b_{i,j}v^{(j)}$.
By assumption on the form of $D$, we have that $b_{i,j} = 0$ when $j < i$, and by choice of $i$ we have that $v^{(j)} = 0$ when $j > i$.
So $R_i v = b_{i,i}v^{(i)} = M_i v^{(i)}$.
By construction, for any vector $w$ we have that $(M w)^{(i)} = R_i w$.
Thus we conclude that $(M v)^{(i)} = M_i v^{(i)}$.
By a similar argument, 
\begin{equation}\label{no backsies}
\mbox{if $j > i$, then $(M v)^{(j)} = 0$.}
\end{equation}

We proceed by induction on $\ell$.
First, suppose that $\ell = 1$.
By definition, we have that $(M - \lambda I)^\nu v = 0$, and so $\left((M - \lambda I)^\nu v\right)^{(i)} = 0$.
Then,
\begin{eqnarray*}
	0 &=& \left((M - \lambda I) v\right)^{(i)}\\
	  &=&  (M v)^{(i)} - \lambda v^{(i)} \\
	  &=& (M_i - \lambda) v^{(i)}, 
\end{eqnarray*}
and so $v^{(i)}$ is an eigenvector of $M_i$ as claimed in the statement of the lemma.
Now suppose that $\ell > 1$.

Let $v' = (M - \lambda I)v$, so that $v'$ is a $\lambda$-eigenvector for $M$ with index $\ell - 1$.
We claim that $(v')^{(i)}$ is a generalized $\lambda$-eigenvector for $M_i$ with index at most $\ell - 1$.
Let $i'$ be the largest index such that $v'^{(i')} \neq 0$.
By (\ref{no backsies}), we have that $i' \leq i$. 
If $i' = i$, then the claim follows from induction.
If $i' < i$, then $(v')^{(i)}$ has index $0$ and the claim follows.

Therefore 
\begin{eqnarray*}
 (M_i - \lambda I)^{\ell} v^{(i)} 	& = &	(M_i - \lambda I)^{\ell-1} \left( (M_i - \lambda I) v \right)^{(i)} \\
					& = &   (M_i - \lambda I)^{\ell} (v')^{(i)} \\
					& = &   0 .
\end{eqnarray*}
\end{proof}

A symmetric argument gives a similar result to Lemma \ref{original to components} for generalized left eigenvectors, with the change that $i$ should be minimized instead of maximized.
The following corollary is then a simple application of Lemma \ref{original to components} with an understanding of how Frobenius normal form orders the irreducible components.

\begin{corollary} \label{up set, down set}
Let $M$ be a general matrix with associated digraph $D$ that has irreducible components $D_1, \ldots D_t$.
Let $M_1, \ldots, M_t$ be the submatrices of $M$ corresponding to $D_1, \ldots D_t$.
Let $v_R$ and $v_L$ be right and left generalized $\lambda$-eigenvectors of $M$.
\begin{itemize}
	\item If $v_R$ is nonzero in coordinate $u$, then there exists a path from $u$ to $w$, where $w \in D_i$, $\lambda$ is an eigenvalue of $M_i$, and $v_R$ is nonzero at $w$.
	\item If $v_L$ is nonzero in coordinate $u$, then there exists a path from $w$ to $u$, where $w \in D_j$, $\lambda$ is an eigenvalue of $M_j$, and $v_L$ is nonzero at $w$.
\end{itemize}
\end{corollary}

\subsection{Constructive Results}

\begin{theorem} \label{aperiodic reducible}
Let $D$ be a digraph with spectral radius $\rho$ and irreducible components $D_1, \ldots D_t$ whose respective adjacency matrices are $A$ and $A_1, \ldots, A_t$.
Suppose $A_1$ is the unique irreducible component with spectral radius $\rho$ (we are not assuming $D$ is in Frobenius normal form here).
Let $v_L, v_R$ be left, right $\rho$-eigenvectors of $A$, normalized such that $v_L v_R = 1$.
If $D_1$ is aperiodic, then for each row vector $w_L$ and column vector $w_R$ there exists $C, \epsilon > 0$ such that 
$$ w_L A^n w_R = \left( (w_L v_R)(v_L w_R) + p(n) \right) \rho^n, $$
where $p(n) < C ( 1 - \epsilon)^n$.
\end{theorem}
\begin{proof}
Because $D_1$ is aperiodic, we have that $D_1$ is primitive.
So by Perron-Frobenius theory, we know that $\rho$ is the unique dominant eigenvalue of $A_1$.
From Lemma \ref{original to components} we know that $v_L^{(1)}$ and $v_R^{(1)}$ are $\rho$-eigenvectors of $A_1$, and by Perron-Frobenius they are the unique dominant eigenvectors of $A_1$.
If there is another left $\rho$-eigenvector $v_*$ of $A$, then $v_*^{(1)} = v_L^{(1)}$.
But then $v_* - v_L$ is a $\rho$-eigenvector of $A$ and $(v_* - v_L)^{(1)}$ is all zeroes, which contradicts Lemma \ref{original to components} and the choice of $A_1$.
This contradiction implies that $v_L$ and $v_R$ are the unique  dominant eigenvectors of $A$.

We claim that we can normalize $v_L$ by a non-zero constant so that $v_L v_R = 1$, as we have assumed.
We will show that $v_L v_R \neq 0$ in two steps: that $v_L^{(1)}v_R^{(1)} \neq 0$ and that $v_L v_R = v_L^{(1)}v_R^{(1)}$.
Perron-Frobenius Theorem to $D_1$ says that $v_L^{(1)}$ and $v_R^{(1)}$ have positive real values in all entries.
Therefore $v_L^{(1)}v_R^{(1)} \neq 0$.
Because $D_1$ is the unique irreducible component with $\rho$ as an eigenvalue, by Corollary \ref{up set, down set} 
\begin{itemize}
	\item if $v_R^{(j)} \neq 0$ then the set of $(D_j, D_1)$-walks is non-empty, and
	\item if $v_L^{(j)} \neq 0$ then the set of $(D_1, D_j)$-walks is non-empty.
\end{itemize}
Recall that when $i \neq j$, we have that the set of $(D_i, D_j)$-walks or the set of $(D_j, D_i)$-walks is empty.
Therefore when $j > 1$ we have that $v_L^{(j)} v_R^{(j)} = 0$.
Moreover, $v_L v_R = \sum_j v_L^{(j)}v_R^{(j)} =  v_L^{(1)}v_R^{(1)} \neq 0$.
Thus the claim is true, and we may assume $v_L v_R = 1$.

At this point forward we may follow the proof of Proposition \ref{primitive behaves nicely}.
\end{proof}

Recall that if $p$ is the period of irreducible digraph $D$, then $D^p$ has $p$ components corresponding to the periodic classes of $D$, and each component induces a primitive digraph.
We are concerned with the digraph $D^p$ when $D$ is not irreducible.

\begin{definition} \label{irreducible dominant def}
Let $D$ be a digraph with an irreducible components $D_1, \ldots D_t$ whose respective adjacency matrices are $A$ and $A_1, \ldots, A_t$.
Fix an index $i$, and let $D_i$ have period $p_i$ and periodic classes $P_{i,1}, \ldots, P_{i,p_i}$.
Let $V_{i,j}$ denote the set of vertices $w$ such that in $D^p$, the set of $(\{w\},P_{i,j})$-walks or the set of $(P_{i,j},\{w\})$-walks is non-empty.
(We allow for walks of length $0$, and so $P_{i,j} \subseteq V_{i,j}$).
For $1 \leq j \leq p_i$, let $D^{\langle i,j \rangle}$ denote the $V_{i,j}$-mask of $D^{p_i}$.
Let $A^{\langle i,j \rangle}$ be the adjacency matrix for $D^{\langle i,j \rangle}$.
\end{definition}

We like to think of the $D^{\langle i, j \rangle}$ ranging over values of $j$ as a partition of the subset of biinfinite walks in $D^{p_i}$ that include a vertex in $D_i$.
Under certain assumptions, the $D^{\langle i, j \rangle}$ for all $i$ and $j$ would then be a partition of the edge shift of $D^{p_i}$.

The ``partition'' of our space has been done carefully.
It is clear that each irreducible component in $D^{\langle i,j \rangle}$ is an irreducible component in $D^{p_i}$.
In the following claim, we show that the edge shifts of $D^{\langle i, j \rangle}$ and $D^{\langle i, j' \rangle}$ have small intersection when $j \neq j'$.

\begin{claim} \label{seperate sets}
We use notation as in Definition \ref{irreducible dominant def}.
If $j \neq j'$, then $P_{i,j} \cap V_{i,j'} = \emptyset$.
\end{claim}
\begin{proof}
By way of contradiction, let $u \in P_{i_j} \cap V_{i,j'}$.
By symmetry, we may assume that $u \in P_{i,j}$, $v \in P_{i,j'}$ and there exists a $(u,v)$-walk $w = w_1,w_2, \ldots, w_k$ in $D^p$.
By Perron-Frobenius theory, we know that $w \not\subseteq (D_i)^p$, and therefore there exists $\ell$ such that $w_\ell \notin V(D_i)$.
By definition of $D^p$, there exists a $(u,v)$-walk in $D$ as 
$$w' = w_1,w_{1,2},w_{1,3}, \ldots,w_{1,p-1},w_2,w_{2,1}, \ldots, w_{k-1,p-2}, w_{k-1,p-1}, w_k.$$
Walk $w'$ implies that in $D$ the set of $(\{w_\ell\}, V(D_i))$-walks and the set of $(V(D_i), \{w_\ell\} )$-walks are non-empty.
But because $w_\ell \notin V(D_i)$, this contradicts that as an irreducible component, $D_i$ is a maximal set that is irreducible.
\end{proof}

At this point, we wish to give an intuitive explanation for what the $V_{i,j}$ represent.

A dominant eigenvector of any matrix $M$ can be found through the \emph{power method}: if $u$ is not orthogonal to the dominant eigenspace, then as $m$ grows $u M^m / \|u M^m \|$ will converge to a vector in the dominant eigenspace.
This intuitively explains why a positive real matrix $A$ has dominant positive real eigenvectors: if we pick $u$ from the high-dimensional space of positive real vectors, then $u M^m / \|u M^m \|$ will be positive and real for all $m$ (there are many issues we are ignoring here; our only goal for this and the next paragraph is intuition).
Now consider how an adjacency matrix $A'$ acts on a row vector $u$: the value in coordinate $w$ of $u A'$ is the value in coordinate $v$ in vector $u$ sum over arcs  $(v,w)$ in the digraph represented by $A'$.
For a fixed $i$ and $j$ as in Definition \ref{irreducible dominant def}, we start with vectors that are positive real in each vertex of $P_{i,j}$ and $0$ everywhere else.
By this interpretation, we see that the vector $u (A')^m$ is nonzero in coordinate $w$ if and only if the set of $(P_{i,j}, \{w\})$-walks of length $m$ in $A'$  is non-empty.
Our definition of $V_{i,j}$ is based on the support of the limit as $m$ grows and $A' = A^{p_i}$.

The sequence of vectors $u M^m / \|u M^m \|$ converges to some dominant eigenvector $v_L^{\langle i,j \rangle}$, but by the nature of $A^{p_i}$ we notice that $v_L^{\langle i,j \rangle}$ is only nonzero in coordinates $w$ such that the set of $(P_{i,j},\{w\})$-walks is non-empty in $A^{p_i}$. 
By a symmetric argument, we intuitively believe that the dominant right eigenvector $v_R^{\langle i,j \rangle}$ that is converged to by starting with nonzero entries only in $P_{i,j}$ has support that is limited to coordinates $w$ such that the set of $(\{w\},P_{i,j})$-walks is non-empty in $A^{p_i}$. 
The $V_{i,j}$ is then defined to be the union of what we think is the support of $v_L^{\langle i,j \rangle}$ and $v_R^{\langle i,j \rangle}$.
Moreover, the nonzero values of $A^{\langle i,j \rangle}$ represent the set of arcs involved in a nonzero summation during the power method.
By this argument, we would expect that $v_L^{\langle i,j \rangle} A^{p_i} = v_L^{\langle i,j \rangle} A^{\langle i,j \rangle}$.

We are doing this because the $A^{\langle i,j \rangle}$ are our digestible chunks of $A$.
Specifically, we will show that they satisfy the assumptions of Theorem \ref{aperiodic reducible}.
And by our above intuition, we expect that the $A^{\langle i,j \rangle}$ will behave like $A$ over a subset of the edge shift of $D$ as partitioned by the $V_{i,j}$.
Our next claim rigorously relates the behavior  of $A^{\langle i,j \rangle}$ to the behavior of $A$ on the desired subspace.

\begin{claim} \label{from subspace to full space}
We use notation as in Definition \ref{irreducible dominant def}.
Suppose $\rho$ is the spectral radius of $A$ and $A_i$.
Furthermore, assume that the set of $\left(V(D_a),V(D_b)\right)$-walks is empty when $\rho$ is an eigenvalue of $A_a$ and $A_b$.
If $v_L^{\langle i,j \rangle}, v_R^{\langle i,j \rangle}$ are left, right $\rho^{p_i}$-eigenvectors of $A^{\langle i,j \rangle}$, then they are left, right $\rho^{p_i}$-eigenvectors of $A^{p_i}$.
\end{claim}
\begin{proof}
We prove this for $v_L^{\langle i,j \rangle}$; the other case is symmetric.
Consider how an adjacency matrix $A'$ acts on a row vector $u$: the value in coordinate $w$ of $u A'$ is the sum of the value in coordinates $v$ of vector $u$ for each arc $(v,w)$ in the digraph represented by $A'$.
The adjacency matrix $A^{\langle i,j \rangle}$ is the adjacency matrix $A^{p_i}$ with $0$ put in entries involving vertices outside of $V_{i,j}$.
If $u$ is a row vector whose support is restricted to $V_{i,j}$ (an assumption that applies to $v_L^{\langle i,j \rangle}$) and $w \in V_{i,j}$, then the $w$ coordinate of $u A^{p_i}$ and the $w$ coordinate of $u A^{p_i}$ are the same, as they are a sum of terms in coordinate $v$ for arcs $(v,w)$ based on two cases: (1) by assumption the term in coordinate $v$ is $0$ if $v \notin V_{i,j}$, and (2) the arc is equally represented by $A^{p_i}$ and $A^{\langle i,j \rangle}$ if $v \in V_{i,j}$.

So to prove the claim, we need to show that if $u$ is a row vector whose support is restricted to $V_{i,j}$ and $w \notin V_{i,j}$, then the $w$ coordinate of $v_L^{\langle i,j \rangle} A^{p_i}$ and the $w$ coordinate of $v_L^{\langle i,j \rangle} A^{\langle i,j \rangle}$ are the same.
It is clear that the $w$ coordinate of $v_L^{\langle i,j \rangle} A^{\langle i,j \rangle}$ is $0$ in this case, as it is outside of the restricted domain of $V_{i,j}$.
We will show that if there exists an arc $(v,w)$ in $D^{p_i}$ such that $v_L^{\langle i,j \rangle}$ is nonzero in coordinate $v$, then $w \in V_{i,j}$, which will prove the claim.

Each irreducible component in $D^{\langle i,j \rangle}$ is an irreducible component in $D^{p_i}$, and so by assumption $D^{\langle i,j \rangle}$ contains exactly one irreducible component with $\rho^{p_i}$ as an eigenvalue.
The support of that irreducible component is $P_{i,j} \subset D_i$.
By Corollary \ref{up set, down set}, if $v_L^{\langle i,j \rangle}$ is nonzero in coordinate $v$, then the set of $(P_{i,j}, \{v\})$-walks in $D^{\langle i,j \rangle}$ is non-empty.
Let $v' = w_1, \ldots, w_\ell = v$ be a $(P_{i,j}, \{v\})$-walk (so $v' \in P_{i,j}$); in combination with arc $(v,w)$ we then have a $(P_{i,j}, \{w\})$-walk: $w_1, \ldots, w_\ell, w$.
\end{proof}

\begin{theorem} \label{incomparable dominant}
We use the notation as in Definition \ref{irreducible dominant def}.
Let $\rho$ be spectral radius of $A$, $s$ the number of irreducible components with spectral radius $\rho$, and order the irreducible components such that $\rho$ is the spectral radius of $A_i$ if and only if $i \leq s$.
Let $P = \prod_{i=1}^s p_i$.
For $1 \leq i \leq s$, let $v_L^{[ i ]}, v_R^{[ i ]}$ be left, right $\rho$-eigenvectors of $A_i$ with all real values normalized such that $v_L^{[ i ]} v_R^{[ i ]} = p_i$.
For $1 \leq i \leq s$ and $1 \leq j \leq p_i$, let $v_L^{\langle i,j \rangle}, v_R^{\langle i,j \rangle}$ be left, right $\rho^{p_i}$-eigenvectors of $A^{\langle i,j \rangle}$, normalized such that $v_L^{\langle i,j \rangle}$ and $v_R^{\langle i,j \rangle}$ give the same values as $v_L^{[ i ]}, v_R^{[ i ]}$ over the domain $D_{i}$.
If for all $1 \leq a,b \leq s$ the set of $\left(V(D_a),V(D_b)\right)$-walks is empty, then for each row vector $w_L$, column vector $w_R$, and integer $k$ there exists $C, \epsilon > 0$ such that 
$$ w_L A^{Pm + k} w_R = \left(q(m) + \sum_{i=1}^s \sum_{j =1 }^{p_i} (w_L v_R^{\langle i,j \rangle})(v_L^{\langle i,j+k \rangle}w_R) \right) \rho^{Pm+k}, $$
where $q(m) < C ( 1 - \epsilon)^m$.
\end{theorem}
\begin{proof}
For $1 \leq i \leq s$, let $V_i$ denote the set of vertices $w$ such that the set of $(\{w\},D_i)$-walks or the set of $(D_i,\{w\})$-walks is non-empty.
Let $D^{\langle i \rangle}$ denote the subdigraph of $D$ induced on $V_i$.
By construction, we have that $\left(D^{\langle i \rangle}\right)^{p_i} = \cup_j D^{\langle i,j \rangle}$.
Let $A^{\langle i \rangle}$ be the adjacency matrix for $D^{\langle i \rangle}$.
Let $D^{\langle * \rangle}$ be the subdigraph of $D$ induced on $\sum_{i > s} D_i$, and let $A^{\langle * \rangle}$ be the adjacency matrix of $D^{\langle * \rangle}$.

The proof of Theorem \ref{incomparable dominant} is broken into three main claims.
In the first claim, we state that the behavior of $D$ is dominated by the behavior of the individual $D^{\langle i \rangle}$ summed across $1 \leq i \leq s$.
In the second claim, we state that the behavior of $(D^{\langle i \rangle})^{p_i}$ is dominated by the behavior of the individual $D^{\langle i,j \rangle}$ summed across $1 \leq j \leq p_i$.
In the third claim, we state that the $v_L^{\langle i,j \rangle}, v_R^{\langle i,j \rangle}$ behave in the same manner as the $v_L^{(i)}, v_R^{(i)}$ in Proposition \ref{periodic behaves nicely} when acted on by $A$.
The proof concludes with the trick from Proposition \ref{periodic behaves nicely} using $w_* = A^kw_R$.

First, we claim that there exists a function $q'$ with constants $C', \epsilon' > 0$ such that $q'(m) < C'(1-\epsilon')^m$ for all $m$, and 
$$ \left\| w_L A^{Pm + k} w_R - w_L \sum_{i=1}^s \left(A^{\langle i \rangle} \right)^{Pm + k} w_R \right\| < q'(m)\rho^{pn+k}.$$
Recall that $v_I^TA^mv_F$ is the number of walks from $I$ to $F$ of length $m$; so the $(i,f)$ entry of $A^{pm + k}$ is the number of walks from vertex $i$ to vertex $f$ of length $pm + k$.
The formula for $ w_L A^{pm + k} w_R$ is linear; so if we are correct for each entry of $A^{pm + k}$ then we are correct overall.
The presentation of the proof to our first claim will be considerably easier in the notation of walks in a digraph.

The walks counted by $\| A^{Pm + k} - \sum_{i=1}^s \left(A^{\langle i \rangle} \right)^{Pm + k} \|$ fall into two categories: (1) walks in $D^{\langle * \rangle}$ and (2) walks in $D^{\langle j \rangle} \cap D^{\langle j' \rangle}$.
By our assumption that for all $1 \leq a,b \leq s$ the set of $\left(V(D_a),V(D_b)\right)$-walks is empty, we see that (2) is a sub-condition of (1).
Let $P_i(x)$ be the characteristic function of $A_i$, so that if $P(x)$ is the characteristic function of $A$, then $P(x) = \prod_{i=1}^t P_i(x)$.
The irreducible components of $D^{\langle * \rangle}$ are exactly $D_{s+1}, \ldots, D_t$ by construction, and so if $P_*(x)$ is the characteristic function of $D^{\langle * \rangle}$, then $P_*(x) = \prod_{i=s+1}^t P_i(x)$.
Let $\lambda_*$ be the largest root of $P_*(x)$, and so $\lambda_* < \rho$.
Let $\epsilon' = 0.5(\lambda - \lambda_*)$ and $A_*$ be the adjacency matrix for $D^{\langle * \rangle}$.
It is well known (for example, see \cite{Parker_Yancey_Yancey}) that $\limsup_{m \rightarrow \infty} m^{-1} \log( w_L A_*^m w_R ) \leq \lambda_*$ for fixed vectors $w_L, w_R$.
The claim therefore follows.

The next claim that we wish to argue is that for each $1 \leq i \leq s$ we have 
$$ \left\| w_L \left(A^{\langle i \rangle}\right)^{p_i m} w_R - w_L \sum_{j=1}^{p_i} \left(A^{\langle i, j \rangle} \right)^{m} w_R \right\| < q''(m)\rho^{pn+k},$$
where $q''(m) \leq C''(1 - \epsilon '')^m$ for some $C'', \epsilon'' > 0$.
Recall that $\left(D^{\langle i \rangle}\right)^{p_i} = \cup_j D^{\langle i,j \rangle}$, so the walks counted by $\|\left(A^{\langle i \rangle}\right)^{p_i} - \sum_{j=1}^{p_i} \left(A^{\langle i, j \rangle} \right)^{m}\|$ are those contained in $D^{\langle i,j \rangle} \cap D^{\langle i,j' \rangle} $ for some $j \neq j'$.
By Claim \ref{seperate sets}, walks in $D^{\langle i,j \rangle} \cap D^{\langle i,j' \rangle} $ are contained in $D^{\langle * \rangle}$.
We have already made the argument in our first claim that these walks are of a smaller order, and therefore this second claim is concluded.

The third claim we wish to prove is that $v_L^{\langle i,j \rangle} A^k = v_L^{\langle i,j+k \rangle} \rho^k$.
Before we do so, we explain how this final claim implies the theorem.
Let $w_* = A^k w_R$.
By our first two claims, we have that 
\begin{eqnarray*}
	w_L A^{Pm + k} w_R	& = &	w_L A^{Pm} w_* \\
				& = & w_L \sum_{i,j} \left(A^{\langle i, j \rangle}\right)^{mP/p_i} w_* + (q'(m) + q''(m))\rho^{Pm+k}\\  
				& = & \left(q'''(m) + \sum_{i=1}^s \sum_{j =1 }^{p_i} (w_L v_R^{\langle i,j \rangle})(v_L^{\langle i,j \rangle} w_*) \right) \rho^{Pm} + (q'(m) + q''(m))\rho^{Pm+k},
\end{eqnarray*}
where the last equality comes from applying Theorem \ref{aperiodic reducible} to each $D^{\langle i,j \rangle}$.
So the last claim implies the theorem with $q(m) = q'(m) + q''(m) + q'''(m)$.

We wish to show that $v_L^{\langle i,j \rangle} A = v_L^{\langle i,j+1 \rangle} \rho$ for $1 \leq i \leq s$.
Recall that $D^{\langle i,j \rangle}$ is a subgraph of $(D^{\langle i \rangle})^{p_i}$.
By Claim \ref{from subspace to full space}, $v_L^{\langle i,j \rangle}$ are left $\rho^{p_i}$-eigenvectors of $(A^{\langle i \rangle})^{p_i}$.
Because the $v_L^{\langle i,j \rangle}$ have disjoint support on $D_i$ by Claim \ref{seperate sets}, we have that the $v_L^{\langle i,j \rangle}$ form a basis for a subspace of left $\rho^{p_i}$-eigenvectors with dimension $p_i$.

A $\lambda'$-eigenvector of $D_i$ is a $(\lambda')^{p_i}$-eigenvector of $(D^{\langle i \rangle})^{p_i}$.
By construction, the only irreducible component in $D^{\langle i \rangle}$ with $\rho$ as an eigenvalue is $D_i$.
By Perron-Frobenius theory, the eigenvalues $\lambda'$ of $A_i$ such that $\|\lambda'\| = \rho$ are exactly $\lambda' = \rho w$, where $w^{p_i} = 1$, and they each have multiplicity $1$.
Let $w_*$ be a primitive root of $x^{p_i} - 1$, and let $u^{[i, j]}$ be the left $(w^j \rho)$-eigenvector of $A^{\langle i \rangle}$.
The set of $\rho^{p_i}$-eigenvectors of $(A^{\langle i \rangle})^{p_i}$ is a space with dimension $p_i$.
Moreover, we have two basii for this space: the $v_L^{\langle i,j \rangle}$ form the first basis and the $u^{[i, j]}$ form the second basis.

We wish to write a transformation formula between the two basii: $v_L^{\langle i,j \rangle} = \sum_{j' = 1}^{p_i} c_{j, j'} u^{[i, j']}$.
To determine the coefficients $c_{j,j'}$ we project these vectors into the restricted domain $V(D_i)$.
By Lemma \ref{original to components}, each eigenvector from either basii will induce an eigenvector of $(A_i)^{p_i}$.
Moreover, because $D_i$ is an irreducible digraph, we know a characterization of the $\rho^{p_i}$-eigenvectors of $(A_i)^{p_i}$.
First, $u^{[i, 0]}$ restricted to $V(D_i)$ gives us the positive real eigenvector $v_L^{[i]}$.
Moreover, we have that $(u^{[i, j]})^{(t)} = w^{-jt}(u^{[i, 0]})^{(t)}$.
By construction, the $v_L^{\langle i,j \rangle}$ are zero in all coordinates except those in $P_{i,j}$.
The transformation becomes clear then: we have that $v_L^{\langle i,j \rangle} = c_j \sum_{j' = 1}^{p_i} w^{jj'} u^{[i, j']}$.

To calculate $c_j$, note that in our above application of Theorem \ref{aperiodic reducible} we require $v_L^{\langle i,j \rangle}v_R^{\langle i,j \rangle} = 1$.
By our argument in Proposition \ref{periodic behaves nicely}, it follows that $c_j = p_i ^{-1}$ is the correct coefficient to satisfy this constraint.

With this transformation, we can directly calculate
\begin{eqnarray*}
 v_L^{\langle i,j \rangle} A		& = &	\left( p_i ^{-1} \sum_{j' = 1}^{p_i} w^{jj'} u^{[i, j']} \right) A 	\\
					& = & 	p_i ^{-1} \sum_{j' = 1}^{p_i} w^{jj'} \left(u^{[i, j']} A \right) 	\\
					& = & 	p_i ^{-1} \sum_{j' = 1}^{p_i} w^{jj'} u^{[i, j']} \lambda w^{j'}  	\\
					& = & \rho	p_i ^{-1} \sum_{j' = 1}^{p_i} w^{(j+1)j'} u^{[i, j']}    	\\
					& = & \rho v_L^{\langle i,j+1 \rangle} .
\end{eqnarray*}

\end{proof}

The standard comment about reducible digraphs is that perturbation theory implies that a weaker version of most of what is known in Perron-Frobenius theory is true.
For example, this comment implies that every nonnegative real matrix has a nonnegative dominant eigenvector associated to a nonnegative real eigenvalue.
As we have seen, this approach is a rather shallow treatment of such a deep and beautiful area (consider the polynomials $p_i(x)$).
We will use this comment for one last statement.

\begin{remark}
The entries in the eigenvectors in Theorem \ref{incomparable dominant} are nonnegative real values.
\end{remark}


\bibliographystyle{alpha}
\bibliography{MatrixAsymp}

\end{document}